\documentclass[11pt]{article}
\usepackage{amsmath, amsfonts, amssymb}
\usepackage{graphicx, amsthm, amsxtra}
\usepackage[english, activeacute]{babel}
\usepackage{enumerate}
\usepackage{latexsym}
\usepackage{color}

\newtheorem{theorem}{Theorem}
\newtheorem{corollary}[theorem]{Corollary}
\newtheorem{lemma}[theorem]{Lemma}
\newtheorem{proposition}[theorem]{Proposition}
\theoremstyle{definition}
\newtheorem{definition}[theorem]{Definition}
\newtheorem{remark}[theorem]{Remark}

\def \1{{\mathbb 1}}

\begin{document}

\begin{center}
{\LARGE Extending the Choquet theory: Trace convexity.}
\end{center}
\smallskip

\begin{center}
{\Large \textsc{Mohammed Bachir, Aris Daniilidis}}
\end{center}
\bigskip

\noindent \textbf{Abstract.} We introduce the notion of trace convexity for functions and respectively, for subsets of a compact topological space. This notion generalizes both classical convexity of vector spaces, as well as Choquet convexity for compact metric spaces. We provide new notions of trace-convexification for sets and functions as well as a general version of Krein-Milman theorem. We show that the class of upper semicontinuous convex-trace functions attaining their maximum at exactly one Choquet-boundary point is residual and we obtain several enhanced versions of the maximum principle which generalize both the classical Bauer's theorem as well as its abstract version in the Choquet theory. We illustrate our notions and results with concrete examples of three different types. 

\vspace{0.55cm}


\vspace{0.55cm}


\tableofcontents

\section{Introduction.}

The classical Krein-Milmann theorem (\cite{KM}) states that every nonempty
convex compact subset~$K$ of a locally convex vector space $E$ can be
represented as the closed convex hull of its extreme points, that is,
$K=\overline{\mathrm{co}}\,(\mathrm{Ext}(K))$. The result is based on the
Hahn-Banach separation theorem together with Zorn's lemma. It asserts in
particular the existence of extreme points in every convex compact set. An
alternative way to obtain the same conclusion is based on the Bauer maximum
principle (\cite{Bau}), which states that for every nonempty convex
compact set, every upper semicontinuous convex function attains its maximum at some extreme point of that set.\smallskip

A more general version of the Krein-Milman theorem can be obtained via the
Choquet representation theory, in terms of Radon measures on $K$ whose support
is contained in the closure of the extreme points of $K$. This theory gives rise to an abstract definition of convexity, further beyond the framework of
vector spaces, the so-called $\Phi$-Choquet convexity, where $K$ is a compact
metric space and $\Phi$ is a closed subspace of the Banach space
$\mathcal{C}(K)$ of real-valued continuous functions on $K$. Then, a function
$f\in \mathcal{C}(K)$ is called $\Phi$-Choquet convex (see forthcoming Definition~\ref{def-choquet-conv}) if for every $x\in K$ and probability measure $\mu$ on
$K$ the following implication holds:
\[
\forall \phi \in \Phi, \,\,\phi(x)=\int_{K}\phi \,d\mu \quad \Longrightarrow
\quad \text{ }f(x)\leq \int_{K}fd\mu.
\]
In particular, a $\Phi$-Choquet convex function $f$ is, by definition,
continuous. Moreover, an abstract version of Bauer's maximum principle holds
true for these functions defined on a compact metric space, which evokes the
so-called Choquet boundary of $K$.  
\smallskip

In this work, we adopt a geometrical approach to define convexity on a compact
(not necessarily metric) space $K$. Considering the canonical injection $\delta^{\Phi}:K\rightarrow \Phi^{\ast}$, given by $\delta^{\Phi}(x)(\phi)=\phi(x)$, for all $\phi\in\Phi$ we may identify $K$ with its homeomorphic image $\delta^{\Phi}(K)$, which lies in particular into the convex $w^{\ast}$-compact subset $K(\Phi)$ of $\Phi^{\ast}$. Under this identification, we call a set (respectively, a function) \emph{convex-trace}, if it is the trace on $K$ of a convex subset of $K(\Phi)$ (respectively,
the restriction on $K$ of a real-valued convex function on $K(\Phi)$), see Definition~\ref{def_TC-set} (respectively, Definition~\ref{def_convex-trace}). Notice that a convex-trace function does not have to be continuous. Nonetheless, for
every continuous function $f$ we may associate a continuous convex-trace
function $\widehat{f}$, which corresponds to a trace-convexification for $f$. Moreover, we show that continuous convex-trace functions coincide with $\Phi$-Choquet convex functions. Therefore trace-convexity can be seen as an alternative geometric definition for Choquet convexity.\smallskip

Our approach pinpoints a natural extension for Choquet convexity. Indeed, we can consider the class of convex-trace functions that are merely upper semicontinuous, which in addition, is the optimal framework to all results related to the maximum principle. In this work, we establish enhanced versions
of the maximum principe for such functions defined on a compact (not
necessarily metric) space, extending both the classical Bauer's maximum
principle for convex functions on locally convex spaces and its abstract
version in the Choquet theory on compact metric spaces. Moreover, in the
particular case that the compact set $K$ is metrizable, we obtain a genericity
result, whose proof does not require Zorn's lemma and is based on a new
variational principle established recently in \cite[Lemma 3]{Ba}. This new
variational principle is in the spirit of that of Deville-Godefroy-Zizler in
\cite{DGZ} and Deville-Revalski in \cite{DR}, and at the same time,
complementary to them: it does not require the existence of a bump function
(ie. a function with a nonempty bounded support) but in turn, it requires the
set to be compact and metrizable. This assumption applies particularly well when
$\Phi$ is the space of affine continuous functions or the space of harmonic
functions, spaces which do not dispose bump functions. \smallskip

Let us mention, for completeness, that motivated by Choquet convexity, there exists another abstract notion of convexity, called $\Phi$-convexity, studied in works by Ky~Fan~\cite{Ky}, M. W. Grossman~\cite{Gr} and B. D. Khanh~\cite{Kh}. This notion is defined algebraically, based on an abstract definition of segments, and aims to extend from convex sets to compact spaces both the classes of convex and of strict quasi-convex functions. Since we are interested in notions extending convexity of functions in a fully compatible way (in the sense that the definition, when applied to a convex subset of a locally convex space, should yield exactly the class of convex functions and not more than this), we shall not deal with this notion in this work. On the other hand, $\Phi$-convex sets (in the theory of Ky Fan) will turn out to be exactly the trace-convex sets (see Remark~\ref{rem-KyFan}). Still we improve the abstract version of the Krein-Milman theory, via a well-adapted definition of $\Phi$-extreme point, that goes with the spirit of trace-convexity. These points are in general much less than the $\Phi$-extreme points in the theory of Ky Fan.

\smallskip

Throughout this work, all topological spaces will be assumed Hausdorff. We shall systema\-tically denote by $K$ a nonempty compact space (which might be metrizable or not) and by $\Phi$ a closed subspace of $\mathcal{C}(K)$ which separates points and contains the constant functions. The whole theory can also be developed in a more general setting, starting from a completely regular space~$X$ (and defining $K:=\beta X$ to be the Stone-\v{C}ech compactification of $X$), or considering an open dense subset $X$ of a given compact space $K$. Notwithstanding, we shall only adopt this more general setting when we deal with convex-trace sets, in order to discuss properly some examples at the end of the manuscript. \medskip

 The manuscript is organized as follows: \smallskip 
 
In Section~\ref{sec:2}, we review concepts related to Choquet convexity in a topological setting ($K$ is a compact space). In particular we recall the notions of representing measure, of Choquet boundary and of a Choquet convex function. \smallskip 

In Section~\ref{sec:3} we introduce the central notion of this work, that is, the notion of trace-convexity, both for functions (Definition~\ref{def_convex-trace}) and for sets (Definition~\ref{def_TC-set}).  We show that Choquet convexity can be equivalently restated in terms of trace convexity for continuous functions (Corollary~\ref{equiv}). This restatement allows a natural extension by considering traces of upper semicontinuous convex functions, see Definition~\ref{def-choquet-conv}. We also introduce convex-trace sets and a trace convexification for sets and establish an abstract Krein-Milman theorem.\smallskip 

In Section~\ref{sec:NEW} we show that this new setting fits perfectly to the framework of Bauer's maximum principle (Theorem~\ref{Choquet1}). Moreover, in the specific case that the compact set is metrizable, an enhanced version of the maximum principle (Theorem~\ref{Pconvex1}) and a generic maximum principle (Theorem~\ref{Choquet}) hold true. \smallskip

Finally, in Section~\ref{sec:ex} we provide three typical examples of different nature to illustrate these notions and the results of this work.  
%\smallskip

\section{Preliminaries and notation.}\label{sec:2}

Let $K$ be a compact topological space and let $\mathcal{C}(K)$ be
the Banach space of continuous real-valued functions on $K$ equipped with the
sup-norm: $\|f\|_{\infty}:=\sup_{x\in K}|f(x)|$, for $f\in\mathcal{C}(K)$. Throughout this work, $\Phi$ will denote a closed subspace of
$\mathcal{C}(K)$ satisfying the following two properties:

\begin{itemize}
\item[(i)] $\Phi$ separates points in $K$ ;\\(that is, for every $x,y\in K$ with $x\neq y$, there exists $\phi \in \Phi$ such that $\phi(x)\neq \phi(y)$)

\item[(ii)] $\Phi$ contains the constant functions.\\
(equivalently, the function $\mathbf{1}(x)=1$, for all $x\in K$ belongs to $\Phi$.)
\end{itemize}
It is well-known that $K$ admits a canonical injection to $\mathcal{C}
(K)^{\ast}$ by means of the following Dirac mapping
\begin{equation}
\left \{
\begin{array}
[c]{l}
\delta:K\longrightarrow \mathcal{C}(K)^{\ast}\smallskip \\
\delta(x)=\delta_{x}\quad \text{with }\delta_{x}(f):=f(x),\text{\thinspace for
all }f\in \mathcal{C}(K)
\end{array}
\right.  \label{eq: delta}
\end{equation}
\newline If we equip $\mathcal{C}(K)^{\ast}$ with the $\sigma(\mathcal{C}
(K)^{\ast},\mathcal{C}(K))$--topology (that we simply call $w^{\ast}
$-topology), then the above injection is homeomorphic and $K$ is topologically
identified to $\delta(K):=\{ \delta_{x}:x\in K\}$ as subset of $(\mathcal{C}(K)^{\ast},w^{\ast})$. We also recall (see \cite{Ru} eg.) that the dual space $\mathcal{C}(K)^{\ast}$ is naturally identified with the Radon measures on $K$ via the duality map
\begin{equation}
\langle \mu,f\rangle=\int_{K}fd\mu,\quad \text{for all }\mu \in \mathcal{C}
(K)^{\ast}\text{ and }f\in \mathcal{C}(K). \label{eq:duality map}
\end{equation}
In particular, $\delta_{x}$ is the Dirac measure of $x$ and (\ref{eq: delta})
becomes:
\[
\delta_{x}(f):=\langle \delta_{x},f\rangle=f(x).
\]
Furthermore, the dual norm $||\mu||_{\ast}$ coincides with the total variation
of the measure $\mu.$\smallskip

We denote by $\mathcal{M}^{1}(K)$ the set of all Borel probability measures on
$K$. This set is a $w^{\ast}$-compact convex subset of $\mathcal{C}(K)^{\ast}$
and coincides with the weak$^{\ast}$ closed convex hull of the set
$\delta(K)$, that is,
\begin{equation}
\mathcal{M}^{1}(K)=\left \{  \mu \in \mathcal{C}(K)^{\ast}:\Vert \mu \Vert_{\ast
}=\langle \mu,\boldsymbol{1}\rangle=1\right \}  =\overline{\mathrm{conv}
}^{w^{\ast}}\left(  \delta(K)\right)  \  \subset \  \mathcal{C}(K)^{\ast},
\label{eq:M1(K)}
\end{equation}
where $\boldsymbol{1}(x)=1,$ for all $x\in K.$

\begin{definition}
[$\Phi$-representing measures]\label{def:Mx}Let $x\in K.$ We say that $\mu
\in \mathcal{M}^{1}(K)$ is a $\Phi$-representing (probability) measure for $x$
if
\[
\phi(x)=\int_{K}\phi \,d\mu,\quad \text{for all }\phi \in \Phi.
\]

\end{definition}

The set of all $\Phi$-representing measures of $x$ is denoted by
$\mathcal{M}_{x}(\Phi).$ Notice that $\delta_{x}\in \mathcal{M}_{x}(\Phi)$, for
every $x\in K$, therefore $\mathcal{M}_{x}(\Phi)$ is nonempty. Every $Q\in
\Phi^{\ast}$ (linear continuous functional on~$\Phi$) can be extended, via the
Hahn-Banach theorem, to an element of $\mathcal{C}(K)^{\ast}$ (linear
continuous functional on $\mathcal{C}(K)$) of the same norm. However this
extension is neither unique nor canonical. Let us consider the following
equivalence relation on $\mathcal{C}(K)^{\ast}:$
\begin{equation}
\mu \sim \mu^{\prime}\; \Longleftrightarrow \; \langle \mu,\phi \rangle=\langle
\mu^{\prime},\phi \rangle,\, \, \text{for all }\phi \in \Phi. \label{eq:equiv}
\end{equation}
We denote by $[\mu]$ the class of equivalence of $\mu \in \mathcal{C}(K)^{\ast}$
under the above binary relation. Since this relation is compatible with the
linear structure of $\mathcal{C}(K)^{\ast},$ setting $\widehat{\pi}(\mu)=[\mu],$ for all $\mu \in \mathcal{C}(K)^{\ast},$ we obtain a linear bounded surjective map 
$\widehat{\pi}:(\mathcal{C}(K)^{\ast},||\cdot||_{\infty})\longrightarrow(\mathcal{C}(K)^{\ast}/\!\!\sim,||\cdot||_{\sim}),$ where
$||\cdot||_{\sim}$ is the quotient norm on $\mathcal{C}(K)^{\ast}/\!\!\sim,$
defined as follows:
\[
||[\mu]||_{\sim}:=\inf \, \left \{  ||\mu^{\prime}||_{\ast}:\mu^{\prime}
\sim \mu \right \}  .
\]
If $\mathcal{C}(K)^{\ast}$ is equipped with its $w^{\ast}$-topology, then we
denote by $\tau$ the final (quotient) topology on $\mathcal{C}(K)^{\ast}/\!\!\sim$
under $\widehat{\pi}$, that is, the finest topology for which the mapping
$$\widehat{\pi}:(\mathcal{C}(K)^{\ast},w^{\ast})\longrightarrow(\mathcal{C}
(K)^{\ast}/\!\!\sim,\tau)$$ is continuous. Therefore, $O\in \tau$ if and only
if $\widehat{\pi}^{-1}(O)$ is $w^{\ast}$-open in $\mathcal{C}(K)^{\ast}$. The
following result shows that the space $(\mathcal{C}(K)^{\ast}/\!\!\sim,\tau)$ is
in fact linearly homeomorphic to $\Phi^{\ast}$, if the latter is considered
with its $\sigma(\Phi^{\ast},\Phi)$-topology (which will be also denoted by
$w^{\ast}$ if no confusion arises). Before we proceed, we observe that the
linear surjective map
\[
\left \{
\begin{array}
[c]{l}
i^{\ast}:\mathcal{C}(K)^{\ast}\longrightarrow \Phi^{\ast} \medskip \\
i^{\ast}(\mu)=\mu|_{\Phi}
\end{array}
\right.
\]
is ($w^{\ast}$-$w^{\ast}$)-continuous, being the adjoint of the identity map $i:(\Phi,\Vert \!\cdot\! \Vert_{\infty})\rightarrow
(\mathcal{C}(K),\Vert\! \cdot\! \Vert_{\infty})$ (which is a linear isometric injection). 

\begin{lemma}
[Identification of $C(K)^{\ast}/\!\!\sim$ with $(\Phi^{\ast},w^{\ast})$
]\label{Lem_topo} The bijective map
\[
\left \{
\begin{array}
[c]{l}
\mathcal{J}:(\mathcal{C}(K)^{\ast}/\!\!\sim,\tau)\longrightarrow(\Phi^{\ast
},w^{\ast})\medskip \\
\mathcal{J}([\mu])=\mu|_{\Phi}
\end{array}
\right.
\]
is a linear homeomorphism between $(\mathcal{C}(K)^{\ast}/\!\!\sim,\tau)$ and
$(\Phi^{\ast},w^{\ast}).$ Moreover, we have the identity:
\begin{equation}
\widehat{\pi}(\mu)=\mathcal{J}^{-1}\circ i^{\ast}(\mu),\quad \text{for every
}\mu \in \mathcal{C}(K)^{\ast}. \label{eq:i*}
\end{equation}

\end{lemma}

\begin{proof} It is straightforward to see that the mapping $\mathcal{J}$ is
a linear bijection and (\ref{eq:i*}) holds. Since $i^{\ast}:(\mathcal{C}(K)^{\ast},w^{\ast})\longrightarrow(\Phi^{\ast},w^{\ast})$ is continuous, and
$\mathcal{J}\circ \widehat{\pi}=i^{\ast}$ it follows from the definition of the
final topology $\tau$ that $\mathcal{J}$ is $(\tau$,$w^{\ast})$-continuous. It
remains to prove that $\mathcal{J}$ maps $\tau$-closed sets to $w^{\ast}$-closed sets. 
To this end, let $F$ be $\tau$-closed in $\mathcal{C}(K)^{\ast
}/\!\!\sim.$ In view of Banach-Dieudonn\'{e} theorem, it is sufficient to prove
that for every $R>0$ the set $\mathcal{J}(F)\cap \mathbf{\bar{B}}_{R}$ is $w^{\ast}$-closed
in $(\Phi^{\ast},w^{\ast}),$ where $\mathbf{\bar{B}}_{R}=i^{\ast}(\bar
{B}(0,R))$ and $\bar{B}(0,R)$ is the closed ball of $\mathcal{C}(K)^{\ast}$
centered at $0$ and radius $R>0.$ By the Banach-Alaoglou theorem and the continuity
of $\widehat{\pi}$ we deduce that $\widehat{\pi}(\bar{B}(0,R))$ is $\tau
$-compact. It follows that the restriction of $\mathcal{J}$ on the ($\tau
$-compact) set $\widehat{\pi}(\bar{B}(0,R))$ is an homeomorphism between
$\widehat{\pi}(\bar{B}(0,R))$ and the closed ball $\mathbf{\bar{B}}_{R}=i^{\ast}(\bar{B}(0,R))$ of $\Phi^{\ast}$. 
This yields that the set
\[
\mathcal{J}(F)\cap \mathbf{\bar{B}}_{R}=\mathcal{J}(F\cap \bar{B}(0,R))
\]
is $w^{\ast}$-closed in $\Phi^{\ast}$ as asserted.
\end{proof}

\bigskip

Restricting the projection mappings $\widehat{\pi}:\mathcal{C}(K)^{\ast
}\longrightarrow \mathcal{C}(K)^{\ast}\!/\!\sim$ and $i^{\ast}:\mathcal{C}
(K)^{\ast}\longrightarrow \Phi^{\ast}$ on the set of probability measures
$\mathcal{M}^{1}(K)$ we obtain an affine homeomorphic bijection between (the
$\tau$-compact set) $\widehat{\pi}(\mathcal{M}^{1}(K)):=\mathcal{M}^{1}(K)/\!\!\sim$ 
and (the $w^{\ast}$-compact set)
\begin{equation}
i^{\ast}(\mathcal{M}^{1}(K)):=K(\Phi)=\{Q\in \Phi^{\ast}:\Vert Q\Vert=\langle
Q,\boldsymbol{1}\rangle=1\} \  \subset \  \Phi^{\ast}, \label{eq: K(Phi)}
\end{equation}
where we continue to denote by $\langle \cdot,\cdot \rangle$ the duality mapping
between $\Phi^{\ast}$ and $\Phi.$ In particular:

\begin{corollary}[Identification of $(\mathcal{M}^{1}(K)/\!\!\sim
,\tau)$ with $(K(\Phi),w^{\ast})$] \label{topo}
The bijective mapping $$\mathcal{J}:(\mathcal{M}^{1}(K)/\!\!\sim
,\tau)\rightarrow(K(\Phi),w^{\ast})$$ defined by $\mathcal{J}([\mu
]):=\mu_{|\Phi}$ is an affine homeomorphism and
\[
\mathcal{J\circ}\,\widehat{\pi}(\mu)=i^{\ast}(\mu)=\mu|_{\Phi},\quad \text{for
all }\mu \in \mathcal{M}^{1}(K).
\]

\end{corollary}

\noindent We also recall the following universal property of the quotient map
\[
\widehat{\pi}:(\mathcal{M}^{1}(K),w^{\ast})\rightarrow(\mathcal{M}^{1}
(K)/\!\!\sim,\tau)
\]

\begin{itemize}
\item (\textit{factorization lemma}) If $G:(\mathcal{M}^{1}(K),w^{\ast
})\rightarrow Z$ is a continuous map such that $\mu \sim \nu$ implies
$G(\mu)=G(\nu)$ for all $\mu,\nu \in \mathcal{M}^{1}(K)$, then there exists a
unique continuous map $H:(\mathcal{M}^{1}(K)/\!\!\sim,\tau)\rightarrow Z$ such
that $G=H\circ \widehat{\pi}$ (where $Z$ is any topological space).
\end{itemize}

\vskip5mm Combining~\eqref{eq: delta} with~\eqref{eq:i*} we obtain a canonical
injection of $K$ into $\Phi^{\ast}$ as follows:
\begin{equation}
\left \{
\begin{array}
[c]{l}%
\delta^{\Phi}:K\longrightarrow(K(\Phi),w^{\ast})\medskip \\
\delta^{\Phi}=i^{\ast}\circ \delta \quad \text{with }\delta_{x}^{\Phi}%
:=\delta^{\Phi}(x)=i^{\ast}(\delta_{x}).
\end{array}
\right.  \label{eq: detla_Phi}
\end{equation}
Therefore $\delta^{\Phi}$ defines a homeomorphism between $K$ and
$\delta^{\Phi}(K)$. In fact, it is often convenient to identify these spaces:
$K\equiv \delta^{\Phi}(K)$. Under the above notation, $\delta_{x}^{\Phi}(\phi)=\phi(x),$ for all $\phi \in \Phi.$ 
Similarly to~\eqref{eq:M1(K)}, the
$w^{\ast}$-compact convex set $K(\Phi)$ in $\Phi^{\ast}$ coincides with the
$w^{\ast}$-closed convex hull of the set $\{ \delta_{x}^{\Phi}:x\in K\}.$

\begin{definition}
[Choquet boundary]\label{def_Ch-bd}The $\Phi$-Choquet boundary $\partial
_{\Phi}(K)$ of $K$ (or simply Choquet boundary, if no confusion arises) is
defined as follows:
\[
\partial_{\Phi}(K):=\{x\in K:\mathcal{M}_{x}(\Phi)=\{ \delta_{x}\} \}.
\]

\end{definition}

\noindent Denoting by $\mathcal{C}(K)^{\ast}_{+}$ the cone of positive Borel measures on $K$, it follows easily that
\[
x\in \partial_{\Phi}(K)\text{\quad \ if and only if \quad}[\delta_{x}
]\cap \mathcal{M}^{1}(K)=\{ \delta_{x}\}=[\delta_{x}]\cap \mathcal{C}
(K)_{+}^{\ast}\,,
\]
where $[\delta_{x}]$ denotes the equivalent class of the Dirac measure
$\delta_{x},$ that is,
\begin{equation}
\mu \in \lbrack \delta_{x}]\quad \text{if and only if\quad}\langle \mu,\phi
\rangle=\phi(x),\; \text{for all }\phi \in \Phi. \label{eq: equiv-delta}
\end{equation}
In addition, it is well-known (see \cite{LNV}, \cite{Ph}, eg.) that a point
$x\in K$ belongs to the $\Phi$-Choquet boundary of $K$ if and only if the
canonical injection $\delta^{\Phi}$, given in (\ref{eq: detla_Phi}), maps this
point to an extreme point of the $w^{\ast}$-compact convex set $K(\Phi)$, that
is,
\begin{equation} \label{eq:Ch-boundary}
\partial_{\Phi}(K)=\left \{  x\in K:\; \delta_{x}^{\Phi}\in \mathrm{Ext\,}
K(\Phi)\right \}  .
\end{equation}

\subsection{Choquet convexity.} 
\label{ssec:2.1}
Let us now recall the following definition (eg. \cite{ChMe}, \cite{LNV}, \cite[Prop. 3.6]{Ph}).

\begin{definition}
[Choquet convex function]\label{def-choquet-conv}A continuous function
$f\in \mathcal{C}(K)$ is said to be $\Phi$-Choquet-convex (or simply,
Choquet-convex, if no ambiguity arises), if for every $x\in K$ it holds:
\[
f(x)\leq \int_{K}fd\mu,\quad \text{for all }\mu \in \mathcal{M}_{x}(\Phi).
\]
\end{definition}

\noindent The set of all Choquet-convex functions will be denoted by
\[
\Gamma_{\Phi}(K):=\{f:K\rightarrow{\mathbb{R}}\,\, \text{Choquet-convex
function}\} \; \subset \; \mathcal{C}(K).
\]

\noindent\textit{Notation.} For any convex subset $S$ of a locally convex space $E$, we set
$$ \Gamma(S):=\{f:S\rightarrow{\mathbb{R}}\,\, \text{convex continuous
function}\}.$$
We also denote by $\Gamma^{>} (S)$ the set of upper semicontinuous (in short, usc)
convex functions, and by $\Gamma^{<} (S)$ the set of lower
semicontinuous (in short, lsc) convex functions on $S$.
\smallskip

\begin{remark}[Compatibility of Choquet convexity]\label{rem-compat-1}
If $K$ is a convex subset of a locally convex space $E$ and $\Phi=\mathrm{Aff}(K)$ is the set of affine continuous functions on $K$, 
then $\delta^{\Phi}$ is an affine homeomorphic bijection between $K$ and $K(\Phi)$ and we have:
\begin{itemize}
\item[] $f\in \Gamma_{\Phi}(K)$ (Choquet convex) if and only if $f\in\Gamma(K)$ (convex continuous). 
\end{itemize}
\end{remark}
\smallskip

Before we proceed, let us recall from \cite[Key Lemma]{LNV} the following result.

\begin{lemma}
[Key Lemma]\label{key-lemma}For every continuous function on $K$ it holds:
\[
\left \{  \int_{K}fd\mu:\; \mu \in \mathcal{M}_{x}(\Phi)\right \}  \,=\, \left[
\underset{\phi \in \Phi,\, \phi \leq f}{\sup}\phi(x),\, \underset{\phi \in
\Phi,\, \phi \geq f}{\inf}\phi(x)\right] .
\]
Therefore, we deduce:
\[
\inf_{\mu \in \mathcal{M}_{x}(\Phi)}\int_{K}fd\mu\,\,=\underset{\phi \in \Phi
,\, \phi \leq f}{\sup}\phi(x)\,\,\leq\, f(x).
\]

\end{lemma}

It follows directly from Definition~\ref{def-choquet-conv} and the above
Key Lemma that
\[
f\in \Gamma_{\Phi}(K)\quad \Longleftrightarrow \quad f(x)=\inf_{\mu \in
\mathcal{M}_{x}(\Phi)}\int_{K}fd\mu=\underset{\phi \in \Phi,\, \phi \leq f}{\sup
}\phi(x),\quad \text{for all }x\in K.
\]

%%%%%%%%%%%%%%
\section{Extending Choquet convexity on topological spaces.}\label{sec:3}

Using the notation of the previous section, we fix a compact space $K$ and a closed subspace
$\Phi$ of $\mathcal{C}(K)$ satisfying conditions (i) and (ii). We also consider
the $w^{\ast}$-compact convex subset $K(\Phi)$ of $(\Phi^{\ast},w^{\ast})$
defined in (\ref{eq: K(Phi)}).

\subsection{Convex-trace functions.}\label{ssec:3-1}

In the spirit of the definition of Choquet boundary, using the identification
of the topological space $K$ with $\delta^{\Phi}(K)$ and the classical
convexity of functions defined on the affine variety $K(\Phi)$  we introduce a
new notion of convexity for real-valued functions defined on $K$. The main idea is that since $K(\Phi)$ is an affine variety, we can define therein convex functions  (in the classical sense) and consider their traces on $\delta^{\Phi}(K)$. We therefore obtain the class of \textit{convex-trace} functions, which can either be continuous ---case in which we recover the class $\Gamma_{\Phi}(K)$ of Choquet convex functions (Corollary~\ref{equiv})---
or more generally upper (or lower) semicontinuous, yielding a natural extension of Choquet convexity, that can be used to 
extend results related to the generalized Bauer's maximum principle. More precisely, we give the following definition.

\begin{definition}
[Convex-trace functions]\label{def_convex-trace}Let $K$ be a compact
space and $\Phi$ a closed subspace of $\mathcal{C}(K)$ that
separates points in $K$ and contains the constant functions. Let further
$\delta^{\Phi}:K\rightarrow(K(\Phi),w^{\ast})$ be the canonical injection given in (\ref{eq: detla_Phi}). A function $f:K\rightarrow
\mathbb{R}$ is called:

\begin{itemize}
\item[(i)] continuous convex-trace with respect to $\Phi$ (or simply continuous $\Phi$-convex-trace), denoted $f\in T\mathcal{C}(K,\Phi),$ if there exists $F\in \Gamma
\mathrm{\,}(K(\Phi))$ such that
\begin{equation}
f=F\circ \delta^{\Phi}. \label{eq:TC}
\end{equation}

\item[(ii)] usc (respectively, lsc) convex-trace with respect to $\Phi$, or simply, usc (respectively lsc) $\Phi$-convex-trace,
denoted $f\in T\mathcal{C}^{>}(K,\Phi)$ (respectively, $f\in T\mathcal{C}^{<}(K,\Phi)$), if there exists $F\in \Gamma^{>}\mathrm{\,}(K(\Phi))$
(respectively, $\Gamma^{<}\mathrm{\,}(K(\Phi))$) such that (\ref{eq:TC}) holds.
\end{itemize}
\end{definition}

In other words, identifying $K$ with its canonical image $\delta^{\Phi}(K)$ in
$\Phi^{\ast}$, a function $f$ is $\Phi$-convex-trace on $K$ whenever $f$ is
the trace of a (usual) convex function on (the affine variety) $K(\Phi)$.
Since every $\phi \in \Phi \subset \mathcal{C}(K)$ is obviously a linear $w^{\ast
}$-continuous functional on $\Phi^{\ast},$ it follows directly that
\[
\Phi \subset T\mathcal{C}(K,\Phi).
\]
Notice that both Choquet convexity (Definition~\ref{def-choquet-conv}) and
trace-convexity (Definition~\ref{def_convex-trace}) depend on the choice of
the closed subspace $\Phi$ of $\mathcal{C}(K)$. This being said, whenever no
confusion arises, we shall drop $\Phi$ and simply talk about usc, lsc or continuous convex-trace functions on $K$, denoting their class by $T\mathcal{C}^{>}(K)$, $T\mathcal{C}^{<}(K)$ and $T\mathcal{C}(K)$ respectively. \smallskip

\begin{remark}[Compatibility of trace-convexity]\label{rem-compat-2}
Similarly to Remark~\ref{rem-compat-1}, if $K$ is a convex subset of a locally convex space $E$, then taking $\Phi=\mathrm{Aff}(K)$, the convex sets 
$K(\Phi)$ and $K$ can be identified via the affine homeomorphism $\delta^{\Phi}$, and the notion of trace-convexity coincides with the classical convexity on $K$, that is:
$$T\mathcal{C}(K)=\Gamma(K), \quad T\mathcal{C}^{>}(K)=\Gamma^{>}(K)\quad \text{and} \quad T\mathcal{C}^{<}(K)=\Gamma^{<}(K).$$
\end{remark}

\begin{remark}[$\Phi $-stability] \label{rem-Phi-stable} A set of functions $\mathcal{B}\subset \mathbb{R}^{K}$ is called $\Phi$-stable, if 
	$\Phi +\mathcal{B}\subset \mathcal{B}.$ It is
	straightforward to see that $\Phi $, $\mathcal{C}(K)$ and $\mathbb{R}^{K}$
	are $\Phi $-stable. It follows easily by Definition~\ref{def-choquet-conv}
	that $\Gamma _{\Phi }(K)$ is $\Phi $-stable. We leave the reader to verify from Definition~\ref{def_convex-trace} that the sets of (usc, lsc, continuous) convex-trace functions $T\mathcal{C}^{>}(K),$ $T\mathcal{C}^{>}(K)$ and respectively, $T\mathcal{C}(K)$ are $\Phi$-stable.
\end{remark}

%%%%%
\subsection{Trace-convexification $f\mapsto\widehat{f}$.}  \label{ssec:3-2}

Let us fix an arbitrary continuous function $f\in \mathcal{C}(K)$. For each $\mu \in \mathcal{M}^{1}(K)$, we define the $w^{\ast}$-open set
\begin{equation}\label{eq:W}
\mathcal{W}_{\mu,f}:=\{ \nu \in \mathcal{C}(K)^{\ast}:\,|\langle \nu-\mu
,f\rangle|\,<1\}
\end{equation}
and we set
\begin{equation}
\mathcal{O}_{f}:=\mathrm{co}\, \left(  \underset{\mu \in \mathcal{M}^{1}
(K)}{\bigcup}\mathcal{W}_{\mu,f}\right)  \supseteq \mathcal{M}^{1}
(K),\label{eq:Of}
\end{equation}
where $\mathrm{co}(\cup_{\mu \in \mathcal{M}^{1}(K)}\mathcal{W}_{\mu,f})$
denotes the convex hull of the $w^{\ast}$-open set $\cup_{\mu \in
\mathcal{M}^{1}(K)}\mathcal{W}_{\mu,f}$. We have that $\mathcal{O}_{f}$ is
$w^{\ast}$-open and convex subset of $\mathcal{C}(K)^{\ast}$. We define
$\widehat{f}$ as follows:
\begin{equation}
\left \{
\begin{array}
[c]{l}
\widehat{f}:K\rightarrow \mathbb{R}\medskip \\
\widehat{f}(x)=\, \underset{\mu \in \lbrack \delta_{x}]\cap \mathcal{O}_{f}}{\inf
}\, \int_{K}\,f\,d\mu.
\end{array}
\right.  \label{eq:f-hat}
\end{equation}
It is straightforward to see that $$f(x)=\langle \delta_x,f\rangle \geq \widehat{f}(x),\quad\text{for all }x\in K.$$
Moreover, since $\int_{K}\, \phi \,d\mu=\phi(x)$ for every $\phi \in \Phi$ and every
$\mu \in \lbrack \delta_{x}]$, it follows that 
$$\widehat{\phi}=\phi \quad \text{for every }\, \phi \in \Phi\,.$$ 
We shall now show that $\widehat{f}$ is a convex-trace function and consequently, $\widehat{f}$ can be seen as a \textit{trace-convexification}
of $f$ on $K$. 

\begin{theorem}
[$\widehat{f}$ is convex-trace]\label{thm_f-hat}For every $f\in \mathcal{C}(K)$
there exists a convex $w^{\ast}$-continuous function $F_{f}:(K(\Phi),w^{\ast
})\rightarrow \mathbb{R}$ such that: 
\begin{itemize}
\item[$\mathrm{(i).}$] $\widehat{f}=F_{f}\circ \delta^{\Phi}$ \ $($therefore, $\widehat{f}\in
T\mathcal{C}(K,\Phi))$ ;   \smallskip 
\item[$\mathrm{(ii).}$]  $-(\Vert f\Vert_{\infty}+1)\leq F_{f}(Q)\leq \Vert f\Vert_{\infty
},\quad$for all $Q\in K(\Phi).$
\end{itemize}
\end{theorem}
\smallskip

\begin{proof}  Let $f\in \mathcal{C}(K),$ and let $\mathcal{O}_{f}$ be the
$w^{\ast}$-open convex set defined in (\ref{eq:Of}). We define:
\begin{align}
G_{f}:(\mathcal{O}_{f},w^{\ast})  &  \rightarrow{\mathbb{R}}\nonumber \\
\mu &  \mapsto \, \underset{\nu \in \lbrack \mu]\cap \mathcal{O}_{f}}{\inf}\,
\int_{K} f\,d\nu,
\end{align}
where $[\mu]$ is the class of equivalence of the measure $\mu \in
\mathcal{O}_{f}\subset \mathcal{C}(K)^{\ast}$, according to~\eqref{eq:equiv}.
It is easy to see that:
\begin{equation}
-\Vert f\Vert_{\infty}\,\Vert \mu \Vert_{\ast}-1\,\leq \,G_{f}(\mu)\,\leq \,\int
_{K}f\,d\mu \,\leq \,\Vert f\Vert_{\infty}\Vert \mu \Vert_{\ast},\quad \text{for all
}\mu \in \mathcal{O}_{f}, \label{Bblewo}
\end{equation}
which guarantees that $G_{f}$ is well-defined (it takes finite values) on
$\mathcal{O}_{f}$. Notice also that
\[
(G_{f}\circ \delta)(x):=G_{f}(\delta_{x})=\widehat{f}(x)\leq f(x),\quad
\text{for all }x\in K.
\]

\smallskip

\textit{Step 1.} The function $G_{f}$ is convex on $\mathcal{O}_{f}
$.\smallskip \newline To this end, let $\mu_{1},\mu_{2}\in \mathcal{C}(K)^{\ast
}$ and $t\in(0,1)$ and fix any $\varepsilon>0$. Then there exist $\nu_{i}
\in \lbrack \mu_{i}]\cap \mathcal{O}_{f},$ for $i\in \{1,2\},$ such that
\[
\langle \nu_{i},f\rangle:=\int_{K}\,f\,d\nu \leq G_{f}(\mu_{i})+\varepsilon.
\]
Since $t\nu_{1}+(1-t)\nu_{2}\in \lbrack t\mu_{1}+(1-t)\mu_{2}]\cap
\mathcal{O}_{f}$ (recall that $\mathcal{O}_{f}$ is a convex set) we obtain
\[
G_{f}(t\mu_{1}+(1-t)\mu_{2})\leq \langle t\nu_{1}+(1-t)\nu_{2},f\rangle
=t\langle \nu_{1},f\rangle+(1-t)\langle \nu_{2},f\rangle \leq G_{f}
(\mu_{1})+(1-t)G_{f}(\mu_{2})+\varepsilon.
\]
Since $\varepsilon>0$ is arbitrary, convexity of $G_{f}$ follows.

\medskip

\textit{Step 2.} $G_{f}$ is locally bounded from above on the $w^{\ast}$-open
set $\mathcal{O}_{f}$.\smallskip \newline Let $\mu_{0}\in \mathcal{O}_{f}$ and
define the $w^{\ast}$-open set $\mathcal{V}_{\mu_{0},f}:=\{ \mu \in
\mathcal{O}_{f}:\,|\langle \mu-\mu_{0},f\rangle|\,<1\}.$ Then for every $\mu
\in \mathcal{V}_{\mu_{0}}$ we have%
\[
G_{f}(\mu)\, \leq \, \int_{K}\,f\,d\mu \,<\,||f||_{\infty}\,||\mu_{0}||_{\ast
}+1,
\]
which yields that $G_{f}$ is bounded from above on $\mathcal{V}_{\mu_{0}}.$

\medskip

\textit{Step 3.} $G_{f}$ is convex $w^{\ast}$-continuous.\smallskip \newline
Since $G_{f}$ is convex, takes finite values on $\mathcal{O}_f$ and it $w^{\ast}$-locally bounded
from above, the assertion follows directly from \cite[Theorem 5.42]{AB}
applied to the $w^{\ast}$-open subset $\mathcal{O}_{f}$ of $\mathcal{C}
(K)^{\ast}$ (where $\mathcal{C}(K)^{\ast}$ equipped with its $w^{\ast}
$-topology is considered as a locally convex space).

\medskip

\textit{Step 4.} $\widehat{f}$ is continuous convex-trace.\smallskip \newline
Notice that by its very definition, $G_{f}(\mu)=G_{f}(\mu^{\prime})$, whenever
$\mu \thicksim \mu^{\prime}$ (with respect to (\ref{eq:equiv})). Since $G_{f}$
is $w^{\ast}$-continuous, then its restriction
\[
{G_{f}}_{|\mathcal{M}^{1}(K)}:(\mathcal{M}^{1}(K),w^{\ast})\rightarrow
\mathbb{R}
\]
is also $w^{\ast}$-continuous. Therefore, by the factorization lemma and the
topological quotient, there exists a unique $\tau$-continuous function
$H_{f}:(\mathcal{M}^{1}(K)/\sim,\tau)\rightarrow{\mathbb{R}}$ such that
${G_{f}}_{|\mathcal{M}^{1}(K)}=H_{f}\circ \widehat{\pi}$. It is straightforward t
see that $H_{f}$ is convex, since $G_{f}$ is convex and $\widehat{\pi}$ is affine.
Using Corollary~\ref{topo}, we deduce that the function $F_{f}:(K(\Phi
),w^{\ast})\rightarrow \mathbb{R}$ defined by $F_{f}=H_{f}\circ \mathcal{J}%
^{-1}$ is convex $w^{\ast}$-continuous. Thus, we have that $F_{f}\circ
i^{\ast}={G_{f}}_{|\mathcal{M}^{1}(K)}$ on $\mathcal{M}^{1}(K)$. Moreover, we
deduce from the definition of $\widehat{f}$ in (\ref{eq:f-hat}) that:
\[
F_{f}(\delta_{x}^{\Phi})=(H_{f}\circ \mathcal{J}^{-1})(\delta_{x}^{\Phi}
)=H_{f}([\delta_{x}])=G_{f}(\delta_{x})=\widehat{f}(x),\quad \text{for all
}x\in K.
\]
This shows that $\widehat{f}$ is the trace on $K\equiv \delta^{\Phi}(K)$ of the
convex $w^{\ast}$-continuous function $F_{f}$ as asserted. The inequality in
$(ii)$ follows from the formula (\ref{Bblewo}). \end{proof}

\subsection{Trace-convexity vs Choquet convexity.} \label{ssec:3-3}

In this subsection we establish (see forthcoming Corollary~\ref{equiv}~(i)$\Leftrightarrow$(iv)) that the class of Choquet convex functions coincides with the class of continuous
convex-trace functions. This latter class admits a natural extension to upper semicontinuous functions, which consists of the most natural
framework to state a generalized maximum principle (see Subsection~\ref{ssec:3-4}).\smallskip

At this stage let us recall from~\cite{Ba1} the $\Phi$-conjugate
$f^{\times}:\Phi \rightarrow{\mathbb{R}}$ of a proper bounded from below
function $f:K\rightarrow{\mathbb{R}}\cup \left \{+\infty \right \}$, which is defined as follows:
\begin{equation}
f^{\times}(\phi):=\sup_{x\in K}\,\{ \phi(x)-f(x)\},\hspace{3mm}\phi \in
\Phi.\label{eq: fx}
\end{equation}
The function $f^{\times},$ being defined in a Banach space $\Phi,$ admits a
usual Fenchel defined by
\[
\left \{
\begin{array}
[c]{l}
(f^{\times})^{\ast}:\Phi^{\ast}\rightarrow{\mathbb{R}}\cup \left \{
+\infty \right \}  \medskip \\
(f^{\times})^{\ast}(Q):=\underset{\phi \in \Phi}{\sup}\, \{ \langle Q,\phi
\rangle-f^{\times}(\phi)\},\hspace{3mm}\text{for all }Q\in \Phi^{\ast}.
\end{array}
\right.  
\]
Restricting the above onto $\delta^{\Phi}(K)$ yields a function $f^{\times
\times}:K\rightarrow{\mathbb{R}}\cup \left \{  +\infty \right \}  $ defined as
follows:
\[
f^{\times \times}(x):=\left ((f^{\times})^{\ast}\circ \delta^{\Phi}\right )(x)=\,\sup_{\phi
\in \Phi}\,\left \{ \langle \delta_{x}^{\Phi},\phi \rangle-f^{\times}(\phi)\right \}=\,\sup
_{\phi \in \Phi}\,\left \{ \phi(x)-f^{\times}(\phi)\right \}, \hspace{3mm}\text{for all }x\in K.
\]
It is easily seen that $f^{\times \times}\leq f$. Moreover, we have
$\phi(x)-f^{\times}(\phi)\leq f(x)$, for all $x\in K$. Taking into account that $\Phi$ contains the
constant functions, we deduce from the above and Lemma \ref{key-lemma} (Key
Lemma) that
\begin{equation}
f^{\times \times}(x)=\underset{\phi \in \Phi,\, \phi \leq f}{\sup}\phi(x)=\inf
_{\mu \in \mathcal{M}_{x}(\Phi)}\int_{K}fd\mu.\label{eq: fxx-sup}
\end{equation}
In particular we obtain the following result.

\begin{proposition}[Relation between $\widehat{f}$ and $f ^ {\times\times}$]
\label{times-hat} For every $f\in \mathcal{C}(K)$, we have that
\[
\widehat{f}(x) \leq f^{\times \times}(x)=\underset{\phi \in \Phi,\, \phi \leq f}{\sup}\phi(x)=\inf
_{\mu \in \mathcal{M}_{x}(\Phi)}\int_{K}fd\mu\leq f(x),\quad
\text{for all }x\in K.
\]
Consequently, $f$ is Choquet convex if and only if  $\widehat{f}=f$ and in this case, we have that 
\[
\widehat{f}(x) = f^{\times \times}(x)=\underset{\phi \in \Phi,\, \phi \leq f}{\sup}\phi(x)=\inf
_{\mu \in \mathcal{M}_{x}(\Phi)}\int_{K}fd\mu= f(x),\quad
\text{for all }x\in K.
\]
\end{proposition}

\begin{proof} Since $\mathcal{M}^{1}(K)\subset \mathcal{O}_{f}$ we deduce
that $\mathcal{M}_{x}(\Phi)= [\delta_x]\cap \mathcal{M}^{1}(K) \subset \lbrack \delta_{x}]\cap \mathcal{O}_{f}$ for
all $x\in K.$ Therefore, it follows readily that:
\[
\widehat{f}(x):=\, \underset{\mu \in \lbrack \delta_{x}]\cap \mathcal{O}_{f}}{\inf
}\, \int_{K}fd\mu \, \leq \inf_{\mu \in \mathcal{M}_{x}(\Phi)}\int_{K}fd\mu
\leq \, \langle \delta_{x},f\rangle=f(x),\quad \text{for all }x\in K.
\]
The assertion follows from~\eqref{eq: fxx-sup}. \smallskip \\
For the second part of the proposition: it follows readily from the above inequalities and Definition~\ref{def-choquet-conv} that if $f=\widehat{f}$, then $f$ is Choquet convex. For the converse, suppose that  $f$ is Choquet convex. Then by Lemma~\ref{key-lemma} (Key Lemma) and formula~\eqref{eq: fxx-sup}
we deduce that for every $x\in K$ 
\begin{equation*}
f(x)\leq \inf_{\mu \in \mathcal{M}_{x}(\Phi )}\int_{K}fd\mu =\underset{\phi
	\in \Phi ,\,\phi \leq f}{\sup }\phi (x)=f^{\times \times }(x)\leq f(x),
\end{equation*}
yielding 
\begin{equation*}
f(x)=\underset{\phi \in \Phi ,\,\phi \leq f}{\sup }\phi (x)=f^{\times \times
}(x).
\end{equation*}
Therefore we have:
\begin{eqnarray*}
\widehat{f}(x)&:=& \underset{\mu \in \lbrack \delta_{x}]\cap \mathcal{O}_f}{\inf}\, \int_{K} f  d\mu = \underset{\mu \in \lbrack \delta_{x}]\cap \mathcal{O}_f}{\inf}\, \int_{K} \left(\underset{\phi \in \Phi,\, \phi \leq f}{\sup}\phi \right)  d\mu \medskip \\ 
& \geq& \underset{\mu \in \lbrack \delta_{x}]\cap \mathcal{O}_f}{\inf} \underset{\phi \in \Phi,\, \phi \leq f}{\sup}\int_{K}\, \phi d\mu 
\geq  \, \underset{\phi \in \Phi, \phi \leq
f}{\sup}\, \underset{\mu \in \lbrack \delta_{x}]\cap \mathcal{O}_f}{\inf}\, \int_{K} \phi
d\mu \, \medskip \\
%&=& \, \underset{\phi \in \Phi, \phi \leq  f}{\sup}\, \hat{\phi}(x),\\
&=& \, \underset{\phi \in \Phi,\, \phi \leq f}{\sup}\phi(x) =  f^{\times \times}(x).
\end{eqnarray*}
Thus, $\widehat{f} \geq f^{\times \times}$ and so the equalities hold. \end{proof}

\bigskip

The following corollary resumes the above results and provides a characterization of
Choquet convex functions. In particular the class of Choquet convex functions
coincides with the class of (continuous) convex-trace functions.

\begin{corollary}[$\Gamma_{\Phi}(K) = T\mathcal{C}(K)$]
\label{equiv} Let $f\in \mathcal{C}(K)$. \\
The following are equivalent:\smallskip

\phantom{i}$\mathrm{(i)}.$ $f$ is Choquet convex $($ie. $f\in \Gamma_{\Phi}(K))$.\medskip

\phantom{i}$\mathrm{(ii)}.$ $f(x)=\underset{\phi \in \Phi,\, \phi \leq f}{\sup}\phi(x)=f^{\times \times
}(x)$, for all $x\in K$.\smallskip

$\mathrm{(iii)}.$  $f(x)=\widehat{f}(x)$, for all $x\in K$.\medskip

$\mathrm{(iv)}.$  $f$ is continuous convex-trace $($i.e $f\in T\mathcal{C}(K))$.
\end{corollary}

\begin{proof} The equivalence of $(i),(ii)$ and $(iii)$ follows from
Proposition~\ref{times-hat}. Implication $(iii)\Rightarrow(iv)$ follows
from Theorem~\ref{thm_f-hat}. It remains to prove that $(iv)\Rightarrow
(iii)$. Let $f\in T\mathcal{C}(K).$ Then there exists $F\in \Gamma
(K(\Phi))$ such that $f=F\circ \delta^{\Phi}$. By assigning the value $+\infty$
outside $K(\Phi)$ we extend $F$ to a $w^{\ast}$-lsc convex function $\tilde
{F}\in \Gamma^{<}\,(\Phi^{\ast}).$ We set
\[
\left \{
\begin{array}
[c]{l}
\tilde{F}^{\ast}:\Phi \longrightarrow \mathbb{R}\cup \{+\infty \} \smallskip \\
\tilde{F}^{\ast}(\phi)=\, \underset{Q\in \Phi^{\ast}}{\sup}\, \left \{  \langle
Q,\phi \rangle-\tilde{F}(Q)\right \}  \,
\end{array}
\right.  \quad \text{and\quad}\left \{
\begin{array}
[c]{l}
\tilde{F}^{\ast \ast}:\Phi^{\ast}\longrightarrow \mathbb{R}\cup \{+\infty
\} \smallskip \\
\tilde{F}^{\ast \ast}(Q)=\, \underset{\phi \in \Phi}{\sup}\, \left \{  \langle
Q,\phi \rangle-\tilde{F}^{\ast}(\phi)\right \}  .\,
\end{array}
\right.
\]
Then the classical Fenchel-duality yields that $\tilde{F}=\tilde{F}^{\ast \ast
}$. On the other hand, since $f=F\circ \delta^{\Phi},$ we deduce readily from
(\ref{eq: fx}) that
\[
f^{\times}(\phi):=\sup_{x\in K}\{ \phi(x)-f(x)\}=\, \underset{x\in K}{\sup
}\, \left \{  \langle \delta_{x}^{\Phi},\phi \rangle-\tilde{F}(\delta_{x}^{\Phi
})\right \}  \leq \, \underset{Q\in \Phi^{\ast}}{\sup}\, \left \{  \langle
Q,\phi \rangle-\tilde{F}(Q)\right \}  =\tilde{F}^{\ast}(\phi),
\]
for all $\phi \in \Phi,$ and consequently
\[
f(x)\geq f^{\times \times}(x):=(f^{\times})^{\ast}(\delta_{x}^{\Phi})\geq
\tilde{F}^{\ast \ast}(\delta_{x}^{\Phi})=\tilde{F}(\delta_{x}^{\Phi
})=f(x),\quad \text{for all }x\in K,
\]
yielding that equality holds. Therefore, $f=f^{\times \times}=\widehat{f}$. \end{proof}
 \bigskip

\begin{remark} \label{rem_ch}
$\mathrm{(i)}.$ Since the function $\widehat{f}$
is convex-trace for every $f\in \mathcal{C}(K)$ (cf. Theorem~\ref{thm_f-hat}), Corollary~\ref{equiv}
yields that $\widehat{\widehat{f}}=\widehat{f}$. Consequently, $\widehat{f}$
is Choquet convex for every $f\in \mathcal{C}(K)$\medskip.

\noindent $\mathrm{(ii)}.$ A careful reader might observe that the definition of $\widehat{f}$ depends on the way the neighborhood~$\mathcal{O}_f$ is defined. Indeed, taking any $\alpha >0$ and defining
\begin{equation*}
\mathcal{W}_{\mu ,f,\alpha }:=\left \{\nu \in \mathcal{C}(K)^{\ast }:\,|\langle
\nu -\mu ,f\rangle |\,<\alpha \right \}\quad \text{and }\quad \mathcal{O}_{f,\alpha
}:=\mathrm{co}\,\left( \underset{\mu \in \mathcal{M}^{1}(K)} \bigcup 
\mathcal{W}_{\mu ,f,\alpha} \right) 
\end{equation*}
we obtain 
\begin{equation*}
\widehat{f}^{\alpha }(x)=\,\underset{\mu \in \lbrack \delta _{x}]\cap 
	\mathcal{O}_{f,\alpha }}{\inf }\,\int_{K}\,f\,d\mu\,,\quad \text{for }\,
x\in K.
\end{equation*}
Then for $0<\varepsilon <1<M$ we readily obtain
\begin{equation*}
\widehat{f}^{M}(x)\leq \widehat{f}(x)\leq \widehat{f}^{\varepsilon }(x)\leq
f^{\times \times }(x)\leq f(x),\quad \text{for }x\in K,
\end{equation*}
with equality if and only if $f\in \Gamma_{\Phi}(K)$. Consequently, the
trace-convexification of a continuous function $f$ is generally not
unique, and not equal to its Choquet-convexification $f^{\times \times}$ (this latter is always the biggest possible convexification). Notice that if $f$ is already Choquet-convex, then all of the above convexifications coincide with~$f$.
\end{remark}

\bigskip 

In the following result, we prove that a supremum of Choquet convex functions
is a lsc convex-trace function.

\begin{proposition}\label{Prop_lsc-CT} Let $\{f_{i}\}_{i\in I}\subset \Gamma_{\Phi }(K)$ be a nonempty family of uniformly bounded Choquet-convex functions on a compact space $K$. Then the (bounded)
function $f:=\sup_{i\in I}f_{i}$ is lower semicontinuous convex-trace on $K$.
\end{proposition}

\begin{proof} Let us assume that $\Vert f_i\Vert_{\infty}\leq M$ for all $i\in I$. Then by Corollary~\ref{equiv} and Theorem~\ref{thm_f-hat}(ii), for each $i\in I$,
there exists a convex $w^{\ast }$-continuous function $F_{i}:K(\Phi)\rightarrow \mathbb{R}$ such that $f_{i}=F_{i}\circ \delta ^{\Phi }$ and 
$-(M+1)\leq F_{i}(Q)\leq M$, for all $Q\in K(\Phi)$. Set $F:=\sup_{i\in I}F_{i}$. Then $F$ is $w^{\ast}$-lsc convex function. On the other hand,
for every $x\in K$, we have 
\begin{equation*}
F(\delta _{x}^{\Phi })\,=\,\sup_{i\in i}F_{i}(\delta _{x}^{\Phi })\,=\,\sup_{i\in I}f_{i}(x)=f(x),
\end{equation*}
which, in view of Definition~\ref{def_convex-trace}(ii), yields that $f\in T\mathcal{C}^{<}(K)$. Finally, notice that both functions $F$ and $f$ take their values on $[-(M+1),M]$, therefore they are real-valued. 
\end{proof}

\subsection{Convex-trace sets.}\label{ssec:3.5}

In this subsection we define the notion of trace-convexity for subsets of a
compact topological space. This notion inevitably relates tightly to the
definition of trace-convexity of functions and, as we shall see, it coincides with the notion of $\Phi$-convexity given by Ky Fan (see \cite{Ky}, \cite{Kh} or \cite{Ba}). In fact, the theory can be naturally developed in a more general framework, that of a completely regular (Hausdorff) topological space. We recall that such spaces admit the so-called Stone-\v{C}ech compactification. (We refer to~\cite{Kelley} for definition and properties of these spaces.)\smallskip

More precisely, throughout this section we assume that $X$ is a completely
regular topological space, which is dense to some compact set $K$. If the
compact set $K$ is not explicitely given, then we can always consider $K=\beta
X$ (the Stone-\v{C}ech compactification of $X$).\smallskip

Let further $\Phi$ be a closed subspace of $\mathcal{C}(K)$ containing the
constant functions and separating points in $X$. If $K=\beta X$, then
$\mathcal{C}(K)\equiv \mathcal{C}_{b}(X)$ (Banach space of all continuous
bounded real-valued functions on $X$). We also recall from~\eqref{eq: K(Phi)}
the definition of the set $K(\Phi)\subset \Phi^{\ast}$ and
from~\eqref{eq: detla_Phi} the definition of the canonical injection
$\delta^{\Phi}:K\rightarrow K(\Phi)$.\smallskip

\begin{remark} [Other cases]\label{Rem_open-dense} The forthcoming definition of $\Phi$-trace-convexity as well as all results of this section remain true if $X$ is in particular compact. In this case we have $X=\beta X=K$ and we can simply replace $X$ by $K$ in all statements. Another interesting special case is when the set $X$ is open (and dense) into some given compact set~$K$, see Subsection~\ref{ssec:5.2} (cf. $X=\mathbb{D}$ is the open complex disk).
\end{remark}

We are ready to give the following definition.

\begin{definition} [Trace convex sets]\label{def_TC-set} A closed set $C\subset X$ is called convex-trace with respect to $\Phi$, if there exists a closed convex subset $\widehat{C}$ of $(K(\Phi),w^{\ast})$ such that
 \begin{equation}
	\delta^{\Phi}(C)=\delta^{\Phi}(X)\cap \widehat{C}. \label{eq:TC-set}
 \end{equation}
The set of all convex-trace subsets of $X$ will be denoted by $\mathcal{P}_{TC}(X)$.
\end{definition}

In other words,
\begin{equation}
C\in \mathcal{P}_{TC}(X)\quad \Longleftrightarrow \quad C=X\cap \left(
\delta^{\Phi}\right)  ^{-1}(\widehat{C}),\; \text{for some $w^{\ast}$-closed
	convex set } \widehat{C}\subset K(\Phi) \label{eq:CTS}
\end{equation}

\begin{remark}
Notice that the set $\widehat{C}$ in~\eqref{eq:TC-set} and~\eqref{eq:CTS} can
be taken to be the $w^{\ast}$-closed convex hull of $\delta^{\Phi}(C)$ in
$K(\Phi)$. Notice further that $X$ satisfies trivially~\eqref{eq:TC-set}, therefore it is convex-trace.
\end{remark}

\subsubsection{Trace convexification of a set}

Based on the equivalence given in~\eqref{eq:CTS} we obtain an
alternative characterization of trace-convexity. In what follows, we denote by
$[g\leq r]:=\{x\in X:\,g(x)\leq r\}$ the sublevel set of the function
$g\in \mathbb{R}^{X}$ at $r>0$.

\begin{proposition} [Characterization of convex-trace sets]\label{prop_TC-set} Let $C$ be a closed subset of $X$. \newline The following assertions are equivalent:
	\begin{itemize}
		\item[$\mathrm{(i).}$] $C\in \mathcal{P}_{TC}(X)$ ;
		
		\item[$\mathrm{(ii).}$] There exists a family $\{ \phi_{i}\}_{i\in I}%
		\subset \Phi$ which is uniformly bounded on $K$ and a bounded sequence $\{
		\lambda_{i}\}_{i\in I}\subset \mathbb{R}$ such that
	\begin{equation}
		C=\bigcap_{i\in I}\{x\in X:\, \phi_{i}(x)\leq \lambda_{i}\}. \label{eq:(ii)}
	\end{equation}
		\item[$\mathrm{(iii).}$] There exists $f\in TC^{<}(K)$ (lsc, convex-trace) such that
	\begin{equation}
		C=\{x\in X:\,f(x)\leq0\}. \label{eq:(iii)}%
	\end{equation}
	\end{itemize}
\end{proposition}

\begin{proof} (i)$\Rightarrow $(ii). Let $C\in \mathcal{P}_{TC}(X)$.
Then there exists a $w^{\ast }$-closed convex subset $\widehat{C}$ of
$K(\Phi )\subset \Phi^{\ast }$ such that \eqref{eq:CTS} holds. By the
Hahn-Banach separation theorem (in the locally convex space 
$(\Phi ^{\ast},w^{\ast })$) we deduce that $\widehat{C}$ is the intersection of closed half-spaces $H_{i}:=\{Q\in \Phi ^{\ast }:\langle Q,\phi _{i}\rangle \leq \lambda _{i}\},$ $i\in I$. Each half-space is defined by a linear functional $\phi _{i}$. Since $K(\Phi )$ is $||\cdot ||_{\Phi ^{\ast }}$-bounded in $\Phi^{\ast }$, we may take $||\phi _{i}||_{\Phi }=||\phi _{i}||_{\infty }=1$ for all $i\in I$ and deduce that $\{ \lambda _{i}\}_{i\in I}\subset \mathbb{(-}M,M)$, for some $M>0$. Then we deduce from~\eqref{eq:CTS} that
	\begin{equation*}
	C=X\cap \left( \delta ^{\Phi }\right) ^{-1}(\widehat{C})=X\bigcap
	\bigcap_{i\in I}\left( \delta ^{\Phi }\right)
	^{-1}(H_{i})=\bigcap_{i\in I}\{x\in X:\, \phi _{i}(x)\leq \lambda
	_{i}\}.
	\end{equation*}
	
(ii)$\Rightarrow $(iii). Let us assume that (\ref{eq:(ii)}) holds for
	some uniformly bounded family $\{ \phi _{i}\}_{i\in I}$. Since $\Phi$ contains the constant functions, we can replace $\phi_i$ by $\phi_i - \lambda_i$ (cf. Remark~\ref{rem-Phi-stable}) and observe that $[\phi\leq \lambda_i]=[\phi-\lambda_i\leq 0]$. Then we set $g=\sup_{i\in I}(\phi_i - \lambda_i)$. It follows readily that $\bigcap_{i\in I}\{x\in X:\, \phi _{i}(x)\leq \lambda_i \}=[g\leq 0]$, while by
	Proposition~\ref{Prop_lsc-CT}, we deduce that $g\in TC^{<}(\beta X)$.	
	\smallskip
	
	(iii)$\Rightarrow $(i). Let us now assume that~\eqref{eq:(iii)} holds and
	let $F\in \Gamma ^{<}(K(\Phi ))$ (lsc convex) such that $f=F\circ \delta
	^{\Phi}$. Then $C=X\cap \left( \delta ^{\Phi }\right) ^{-1}(\widehat{C})$
	where $\widehat{C}=[F\leq 0]$ is obviously closed and convex in $K(\Phi)$,
	and the result follows from~\eqref{eq:CTS}. \end{proof}

\bigskip

Using~\eqref{eq:(ii)} of the above proposition, we can easily deduce the following corollary.

\begin{corollary} [Separation theorem]\label{Cor_separation} Let $C$ be a nonempty subset of $X$. \newline The following assertions are equivalent:
	
	\begin{itemize}
		\item[$\mathrm{(i).}$] $C\in \mathcal{P}_{TC}(X)$ (i.e. $C$ is (closed) convex-trace in
		$X$) ;
		
		\item[$\mathrm{(ii).}$] For every $\bar{x}\in K\diagdown C,$ then there exists $\phi
		\in \Phi$ such that
		\begin{equation} 
		\sup_{x\in C}\, \phi(x)<\phi(\bar{x}). \label{eq:sep}
		\end{equation}
		
	\end{itemize}
\end{corollary}

Given a nonempty subset $S\subset X$ and a closed subspace $\Phi$ of
$\mathcal{C}(K)$ as above, we define the $\Phi$-trace convexification
$\overline{\mathrm{co}}_{\Phi}(S)$ of $S$ as follows:
\[
\overline{\mathrm{co}}_{\Phi}(S)=\bigcap_{S\subset C}\left \{  C:\;C\in
\mathcal{P}_{TC}(X)\right \}
\]
The following result follows easily from the definitions.

\begin{proposition}
	[Characterization of convexification]\label{convexification} Let\ $S\subset X$
	and consider the usual closed convexification of $\delta^{\Phi}(S)$ in
	$K(\Phi)$, that is,
	\[
	\overline{\mathrm{co}}^{w^{\ast}}\,(\delta^{\Phi}(S))=\bigcap_{\delta^{\Phi
		}(S)\subset \widehat{C}}\left \{  \widehat{C}:\; \widehat{C}\subset K(\Phi)\,\,\text{ $w^{\ast}$-closed and convex }\right \}  .
	\]
	Then
	\[
	\overline{\mathrm{co}}_{\Phi}(S)=X\cap \left(  \delta^{\Phi}\right)
	^{-1}(\overline{\mathrm{co}}^{w^{\ast}}\,(\delta^{\Phi}(S)).
	\]
	
\end{proposition}

\subsubsection{Abstract Krein-Milman theorem and relation with Ky Fan convexity}

If $X=K$ is compact, property~(ii) of Proposition~\ref{prop_TC-set} corresponds to the definition of $\Phi$-convexity given by Ky Fan~\cite{Ky}. Therefore Proposition~\ref{prop_TC-set} shows that:
\begin{itemize}
\item[--] A set $C\subset K$ is $\Phi$-convex-trace if and only if it is $\Phi$-convex in the sense of Ky Fan. 
\end{itemize} 

Before we proceed, let us recall the classical Krein-Milman theorem in a locally convex space~$E$. For $x,y\in E$, we define the open segment 
$(x,y):=\{tx+(1-t)y:\,t\in (0,1)\}\,\subset\, E$. 
We first recall the definition of an extreme point.

\begin{definition} [Extreme point] \label{def-extreme}
Let $S$ be a nonempty subset of a locally convex space $E.$ We say that
$\bar{p}\in S$ is an \emph{extreme} point of $S$ if whenever $\bar{p}\in
(p_{1},p_{2})$, with $p_{1},p_{2}\in C$, it holds $p_{1}=p_{2}=\bar{p}$. We
denote by $\mathrm{Ext}(S)$ the set of all extreme points of $S$.
\end{definition}

Recall that the Krein-Milman theorem asserts that if $C$ is a convex
compact subset of a locally convex space, then $C$ is the closed convex hull
of its extreme points, that is, $C=\overline{\mathrm{co}}(\mathrm{Ext}(C))$. A
more precise version asserts that for any nonempty subset $A$ of $C$ it holds:
\[
\overline{\mathrm{co}}(A)=\overline{\mathrm{co}}(\mathrm{Ext}(\bar{A})).
\]
The Krein-Milman theorem has a partial converse known as Milman's theorem
(see~\cite{Ph} eg.) which states that if $A$ is a subset of $C$ and the closed
convex hull of $A$ is all of $C$, then every extreme point of $C$ belongs to
the closure of $A$, that is,
\[
(A\subset C;\hspace{2mm}C=\overline{\mathrm{co}}(A))\Longrightarrow
\mathrm{Ext}(C)\subset \overline{A}.
\]
We shall now see that the above results can be naturally stated for $\Phi
$-convex-trace subsets of a completely regular topological space $X.$ To this
end, let us start with the following definition which extends the notion of an
extreme point in this topological setting.

\begin{definition}
	[$\Phi$-extreme point]\label{phi_extrem}Let $X$ be completely regular
	topological space and $S\subset X$. A point $x\in S$ is called $\Phi$-extreme in $S$ if $\delta_{x}^{\Phi}=\delta^{\Phi}(x)$ is an extreme point of
	$\overline{\mathrm{co}}^{w^{\ast}}\,(\delta^{\Phi}(S))$. We denote by
	$\mathrm{Ext}_{\Phi}(S)$ the set of all $\Phi$-extreme points of $S$.
\end{definition}

We shall now show that the Krein-Milman theorem holds true in our abstract setting.

\begin{theorem} [Abstract Krein-Milman theorem]\label{phi_KM} Let\ $S\subset X$ be a compact set. Then,
\[
\overline{\mathrm{co}}_{\Phi}(S)=
\overline{\mathrm{co}}_{\Phi}(\mathrm{Ext}_{\Phi}(S)).
\]
Therefore, if $C\subset X$ is compact and $\Phi$-trace convex, then
$C=\overline{\mathrm{co}}_{\Phi}(\mathrm{Ext}_{\Phi}(C))$. \smallskip \\
In other words, a
compact $\Phi$-convex trace set is the $\Phi$-convex hull of its $\Phi$-extreme points.
\end{theorem}

\begin{proof} Applying the Krein-Milman theorem in the locally convex space
$(\Phi^{\ast},w^{\ast})$ for the convex compact set $C:=\overline{\mathrm{co}}^{w^{\ast}}\,(\delta^{\Phi}(S))\subset K(\Phi)$ we have that
\[
\overline{\mathrm{co}}^{w^{\ast}}\,(\delta^{\Phi}(S))=
\overline{\mathrm{co}}^{w^{\ast}}\,(\mathrm{Ext}(\overline{\mathrm{co}}^{w^{\ast}}\,(\delta^{\Phi}(S))).
\]
On the other hand, by the partial converse of the Krein-Milman theorem
(Milman's theorem), setting $A=\delta^{\Phi}(S)$ we deduce that
\[
\mathrm{Ext}(\overline{\mathrm{co}}^{w^{\ast}}\,(\delta^{\Phi}(S))\subset
\overline{\delta^{\Phi}(S)}^{w^{\ast}}=\delta^{\Phi}(S).
\]
It follows from the definition of $\Phi$-extreme points that
\[
\mathrm{Ext}(\overline{\mathrm{co}}^{w^{\ast}}\,(\delta^{\Phi}(S))=\delta
^{\Phi}(\mathrm{Ext}_{\Phi}(S)).
\]
Using Proposition~\ref{convexification}, the Krein-Milman theorem and the
above equality, we have
\begin{align*}
\overline{\mathrm{co}}_{\Phi}(S)  &  =X\cap \left(  \delta^{\Phi}\right)
^{-1}(\overline{\mathrm{co}}^{w^{\ast}}\,(\delta^{\Phi}(S)))\\
&  =X\cap \left(  \delta^{\Phi}\right)  ^{-1}(\overline{\mathrm{co}}^{w^{\ast}
}\,(\mathrm{Ext}(\overline{\mathrm{co}}^{w^{\ast}}\,(\delta^{\Phi}(S)))))\\
&  =X\cap \left(  \delta^{\Phi}\right)  ^{-1}(\overline{\mathrm{co}}^{w^{\ast}
}\,(\delta^{\Phi}(\mathrm{Ext}_{\Phi}(S))))\\
&  =\overline{\mathrm{co}}_{\Phi}(\mathrm{Ext}_{\Phi}(S)).
\end{align*}
This gives the first part of the theorem. For the second part, if $C$ is
assumed compact and $\Phi$-convex trace, then $C=\overline{\mathrm{co}}_{\Phi
}(C)$ and the conclusion follows from the first part.
\end{proof}

\begin{corollary}\label{cor-ch-bd}
Let $K$ be a compact Hausdorff topological space and $\Phi$ be a closed
subspace of~$\mathcal{C}(K)$ containing the constant functions and separating
points in $K$. Then, we have
\[
K=\overline{\mathrm{co}}_{\Phi}(\partial_{\Phi}(K)).
\]
	
\end{corollary}

\begin{proof} Comparing Definition~\ref{phi_extrem} with~\eqref{eq:Ch-boundary} we easily see that $x\in \mathrm{Ext}_{\Phi}(K)$ if and only if $x\in\partial_{\Phi}(K)$. In other words, the $\Phi$-extreme points of $K$ and the
elements of the Choquet boundary of $K$ are the same. On the other hand, the
set $K$ is trivially $\Phi$-trace convex, since $K(\Phi)$ is convex compact
and $\delta^{\Phi}(K)=\delta^{\Phi}(K)\cap K(\Phi)$. The conlusion is
straightforward from Theorem~\ref{phi_KM}. \end{proof}

\begin{remark} [Comparison with the Ky Fan theory] \label{rem-KyFan} According to the Ky Fan theory (\cite{Ky},	\cite{Kh}), given $y,z\in K,$ the $\Phi$-segment $[y,z]_{\Phi}$ is defined to be the set of all $x\in K$ such that for any $\phi \in \Phi$ the following implication holds:
	\begin{equation}
	\phi(x)\leq \min \{ \phi(y),\phi(z)\} \Longrightarrow \phi(x)=\phi(x)=\phi
	(y).\label{eq:KyF}
	\end{equation}
Then a point $x\in K$ is called $\Phi$-extreme (in the sense of Ky Fan) for
the compact set $K$ if whenever $x\in \lbrack y,z]_{\Phi}$ for $y,z\in K$, it
holds $x=y=z$. We now prove the following claim.\smallskip
	
\noindent \textit{Claim}. Every $\Phi$-extreme point of $K$ (cf. Definition~\ref{phi_extrem}) is $\Phi$-extreme in the sense of Ky Fan.
\smallskip\newline
\noindent \textit{Proof of the Claim.} Indeed, by~\eqref{eq:Ch-boundary} we have $x\in \mathrm{Ext}_{\Phi}(K)\Leftrightarrow \delta_{x}^{\Phi}\in \mathrm{Ext}(K(\Phi))$ (recall that
$K(\Phi)=\overline{\mathrm{co}}^{w^{\ast}}\,(\delta^{\Phi}(K))$). Let us
assume, towards a contradiction, that $x\in \mathrm{Ext}_{\Phi}(K)$ and
$x\in \lbrack y,z]_{\Phi}$ for some $y,z\in K$ with $y\neq x$. Then if $z=x$,
since $\Phi$ separates points in $K$ we get $\phi(x)=\phi(z)<\phi(y)$ for some
$\phi \in \Phi$, contradicting~\eqref{eq:KyF}. If now both $y,z$ are different than~$x$, then in view of~\eqref{eq:Ch-boundary} $\delta_{x}^{\Phi}$ is extreme in $K(\Phi)$ and consequently
$\delta_{x}^{\Phi}\notin \lbrack \delta_{y}^{\Phi},\delta_{z}^{\Phi}]$ (the usual segment in the $w^{\ast}$-compact convex set
	$K(\Phi)\subset \Phi^{\ast}$). By Hahn-Banach theorem (for the $\sigma
	(\Phi^{\ast},\Phi)$-topology of $\Phi^{\ast}$) we deduce that for some
	$\phi \in \Phi$,
	\[
	\min \{ \phi(y),\phi(z)\}=\min \{ \langle \delta_{y}^{\Phi},\phi \rangle
	,\langle \delta_{z}^{\Phi},\phi \rangle \} \geq \min_{Q\in \lbrack \delta_{y}^{\Phi
		}.\delta_{z}^{\Phi}]}\, \langle Q,\phi \rangle>\langle \delta_{x}^{\Phi},\phi \rangle=\phi(x),
	\]
which again contradicts~\eqref{eq:KyF}. This completes the proof of the claim. \smallskip \newline
\noindent The converse of the claim is not true in general, since $\Phi$-extreme points in the sense of Ky Fan may be numerous. To see this, take for instance $X=\mathbb{D}$ to be the unit disk of the complex plane and  $\Phi$ be the class of harmonic functions of the open disk, which are continuous on the closed disk $\mathbb{\bar{D}}$ (see Subsection~\ref{ssec:5.2}). Then we easily see that all $\Phi$-segments are trivial (singletons) and consequently all points of $\mathbb{\bar{D}}$ are extreme (whereas the Choquet boundary of $\mathbb{D}$ coincides with the usual topological boundary).\smallskip\newline 	
Therefore, Theorem~\ref{phi_KM} is an enhanced version of the Ky Fan result
in~\cite{Ky} (see also~\cite{Kh}).
\end{remark}

\section{Maximum principles for convex-trace functions.}
\label{sec:NEW}

In this section we establish a general version of maximum principle, that goes beyond Choquet convexity, and is adapted to the setting of Definition~\ref{def_convex-trace}. In particular: \smallskip

In Subsection~\ref{ssec:3-4} we establish a maximum principle for upper semicontinuous convex-trace functions on a compact topological space (Theorem~\ref{Choquet1}), generalizing the maximum principle obtained in \cite[Section~3.2]{LNV} in a twofold aspect: the function $f$ is not necessarily continuous, and the compact $K$ is not assumed to be metrizable.\smallskip 

In Subsection~\ref{sec:4} we consider the metrizable case and establish enhanced versions of the maximum principle evoking a family of functions as well as a genericity result.  

\subsection{Maximum principle in topological spaces.}
\label{ssec:3-4}

We recall from Corollary~\ref{equiv} that the class of Choquet convex functions coincides with the class of continuous convex-trace functions, while our results are formulated in $T\mathcal{C}^{>}(K)$ (upper semicontinuous convex-trace functions). Our results are based on the classical Bauer maximum principle. \smallskip

Before we proceed, let us introduce the following notation: for a nonempty set $C$ and a function $f:C\to{\mathbb{R}}$, we denote by
\[
C_{\max} (f):=\lbrace \bar x\in C: f(\bar x)=\max_{x\in C} f(x)\rbrace\,,
\]
the set of maximizers of $f$ on $C$. We also denote by $C_{\min} (f):= C_{\max}(-f)$ the set of minimizers of $f$ on $C$. Under this notation we have the following result:

\begin{proposition}[Maximizing a convex function on $K(\Phi)$]
\label{max} Let $K$ be a compact space. Let $F : (K(\Phi), w^{*})
\to{\mathbb{R}}$ be an upper semicontinuous convex function. Then, we have that
\[
\max_{Q \in K(\Phi)} F(Q) \,=\, \max_{x\in K} (F \circ \delta^{\Phi} )(x),
\]
and consequently, 
$$\delta^{\Phi}(K_{\max}( F\circ \delta^{\Phi} ))\subset [K(\Phi)]_{\max}(F).$$
\end{proposition}

\begin{proof}
Using the classical Bauer theorem, we have that
\[
\max_{Q \in K(\Phi)} F(Q) \,=\, \max_{Q \in  \mathrm{Ext}(K(\Phi))}F(Q).
\]
Since $\mathrm{Ext}(K(\Phi))=\delta^{\Phi}(\partial_{\Phi}(K))\subset
\delta^{\Phi}(K)$, it follows that
\[
\max_{Q \in K(\Phi)}F(Q)\, \leq \, \max_{x \in K}(F\circ \delta^{\Phi}) (x).
\]
On the other hand, since $\delta^{\Phi}(K)\subset K(\Phi)$, the above inequality is in fact an equality. Therefore we have that
$\delta^{\Phi}(K_{\max}(f))\subset [K(\Phi)]_{\max}(F)$ as asserted. \end{proof}
\bigskip

We now establish the following result, which extends \cite[Maximum principle (page~241)]{LNV} from the class of Choquet convex functions (which coincides with $T\mathcal{C}(K)$)
to the class of usc convex-trace functions $T\mathcal{C}^{>}(K)$.

\begin{theorem}[Bauer maximum principle for usc convex-trace functions]
\label{Choquet1} Let $K$ be a compact space and $f:K\rightarrow
{\mathbb{R}}$ be an usc convex-trace function. Then, there exists
$\bar x\in \partial_{\Phi}(K)$ such that $f(\bar x)=\max_{x\in K}f(x)$.
\end{theorem}

\begin{proof}
By definition, there exists an upper semicontinuous convex function
$F:(K(\Phi),w^{\ast})\rightarrow{\mathbb{R}}$ such that $f=F\circ \delta^{\Phi}$. Applying the classical Bauer theorem to $F$, there exists $\widehat{Q}\in \mathrm{Ext}(K(\Phi))$ such that
\[
\max_{Q \in K(\Phi)}F(Q)=F(\widehat{Q}).
\]
Since $\mathrm{Ext}(K(\Phi))=\delta^{\Phi}(\partial_{\Phi}(K))$, there exists
$\bar x\in \partial_{\Phi}(K)$ such that $\widehat{Q}=\delta_{\bar x}^{\Phi}$. It follows that
\[
\max_{Q \in K(\Phi)}F(Q)=f(\bar x).
\]
Since $\delta^{\Phi}(K)\subset K(\Phi)$, the inequality $\max_{Q\in K(\Phi)}F(Q) \geq \max_{x\in K}f(x)$ holds trivially. \\
The proof is complete. \end{proof}
 \bigskip

Theorem~\ref{Choquet1} yields directly the following result.
\begin{corollary}
Let $K$ be a compact space and $f:K\rightarrow{\mathbb{R}}$ be an
upper semicontinuous convex-trace function. If $f(x)\leq0$ for all $x\in \partial_{\Phi}(K)$, then
$f(x) \leq 0$, for all $x \in K$.
\end{corollary}

Recalling from \eqref{eq:f-hat} the definition of $\widehat{f}$, and combining Theorem~\ref{Choquet1} with Theorem~\ref{thm_f-hat}, we obtain the
following corollary.

\begin{corollary}
Let $K$ be a compact space and $f\in \mathcal{C}(K)$. Then, there
exists $\bar x\in \partial_{\Phi}(K)$ such that $\widehat{f}(\bar x) = \max_{x\in K}\widehat{f}(x)$.
\end{corollary}

\subsection{The maximum principle for compact metric space.}\label{sec:4}

In this subsection we focus on the case where the compact space $K$ is metrizable, which in fact, is the usual setting for the notion of Choquet convexity, and the framework considered in~\cite{LNV}. In this case, making use of the metric structure of $K$ and of the metrizability of the $w^*$-topology of $K(\Phi)$, and using an adequate version of variational principle, we are going to establish extensions of the Bauer maximum principle in two directions:
\begin{itemize}
 \item[--] We shall deal with the class $T\mathcal{C}^{>}(K)$ of upper semicontinuous convex-trace functions (this class contains strictly the class of Choquet convex functions).
 \item[--] We establish a multi-maximum result evoking a family of functions (Theorem~\ref{Pconvex1}), as well as an abstract generic result (Theorem~\ref{Choquet}).
 \end{itemize}   

Let us start by recalling from \cite[Lemma~3]{Ba} the following version of variational principle that we shall use in the sequel. 

\begin{lemma}[Variational Principle]\label{MBAD}  Let $(K,d)$ be a compact metric
space and $(\Phi,\Vert.\Vert_{\Phi})$ be a Banach space such that 
$\Phi \subset \mathcal{C}(K)$, $\Phi$ separates points in $K$ and for some $\alpha>0$ it holds:
$$ \alpha \, \Vert \phi \Vert_{\Phi}\,\geq\, \Vert \phi \Vert_{\infty}, \quad \text{ for all } \,\phi \in \Phi.$$
Let $f:(K,d)\rightarrow{\mathbb{R}}\cup \left \{  +\infty \right \}  $ be a proper
lower semicontinuous function. Then, the set
\[
N(f)=\big \{  \phi \in \Phi:\,\, K_{\min}(f-\phi) \text{ is not a singleton}  \,\,\big \}
\]
is of first Baire category in $\Phi$.
\end{lemma}

Before we proceed, let us recall the following definition. 

\begin{definition} [$w^{\ast}$-exposed points]\label{def-exp} Let $E$ be a locally convex space and $S$ a nonempty $w^{\ast}$-closed subset of the dual space $E$. We say that $\bar{p}\in S$ is $w^{\ast}$-\emph{exposed} in $S$, and denote $\bar{p}\in w^{\ast}$-$\mathrm{Exp}(S)$, if there exists $x\in X$ such that
\[
\langle \bar{p},x\rangle>\langle p,x\rangle,\quad \text{for all }p\in
S\setminus \{ \bar{p}\}.
\]
	
\end{definition}

It is straightforward from Definition~\ref{def-extreme} and Definition~\ref{def-exp} that every $w^{\ast }$-exposed point is extreme, that is $w^{\ast }\text{-}\mathrm{Exp}(S)\subset \mathrm{Ext}(S)$.
This inclusion might in general be strict. \smallskip \newline
The classical Krein-Milman theorem ensures the existence
of extreme points for convex compact sets. However, in absence of
convexity, we cannot in general conclude that $\mathrm{Ext}(S)\neq \emptyset 
$. Still, the conclusion holds true if $E$ is a Banach space and $S\subset E$
is compact for either the norm or the weak topology. But if $E=X^{\ast }$ is
a dual Banach space and $S$ is $w^{\ast }$-compact, the conclusion could fail.
This being said, using Lemma~\ref{MBAD} we deduce an important instance of $w^{\ast }$-compact sets with extreme points by establishing the existence of 
$w^{\ast }$-exposed points.

\bigskip

\begin{lemma}[existence of extreme points]
	\label{lem-ext} Let $\Phi $ be a Banach space and $S\subset \Phi ^{\ast }$
	be $w^{\ast }$-compact and metrizable. Then $w^{\ast }$-$\mathrm{Exp}(S)\neq
	\emptyset $ and consequently $\mathrm{Ext}(S)\neq \emptyset .$
\end{lemma}

\begin{proof}
Every $\phi \in \Phi $ can be seen as a continuous function on the compact
metric space $(S,w^{\ast })$. Set: $\alpha:=\max_{p\in S}\Vert p\Vert +1$. Then
\begin{equation*}
\Vert \phi \Vert _{\Phi }\,=\,\sup_{p\in \Phi ^{\ast }}\,|\langle \frac{p}{||p||},\phi \rangle |\,\,\geq \,\,\sup_{p\in S}\,|\langle \frac{p}{||p||},\phi \rangle |\,\,\geq \,\frac{1}{\alpha}\,\sup_{p\in S}\,|\langle p,\phi \rangle |\,=\frac{1}{\alpha}\,\Vert
\phi \Vert _{\infty },
\end{equation*}
where the last inequality is based on the linearity of $\phi $. Since
obviously $\Phi $ separates points in $S,$ we can apply Lemma~\ref{MBAD} to
the function $f\equiv 0$ to deduce that for a generic $\phi \in \Phi$, 
$-\phi $ attains a unique minimum on $S$ at some point $\bar{p}\in S$.
This yields that $\bar{p}\in w^{\ast }$-$\mathrm{Exp}(S)\subset \mathrm{Ext}(S)$. Therefore, both $w^{\ast }$-$\mathrm{Exp}(S)$ and $\mathrm{Ext}(S)$ are nonempty.
\end{proof}

\subsubsection{A multi-maximum principle.}
\label{ssec:4.1}
Let us first establish the following result, which has an independent
interest.

\begin{lemma}[Common extreme maximizer]
\label{Pconv} Let $\Phi $ be a Banach space and $C\subset \Phi ^{\ast }$ be convex $w^{\ast }$-compact and metrizable. Let $\{F_{i}\}_{i\in I}$ be a nonempty family of real-valued $w^{\ast }$-usc, convex functions on $(C,w^{\ast })$ with a common maximizer, that is, 
\begin{equation*}
	C_{\max }(I):=\bigcap_{i\in I}C_{\max }(F_{i})\neq \emptyset .
\end{equation*}
Then
\begin{equation*}
	w^{\ast }\text{-}\mathrm{Exp}(C_{\max }(I))\neq \emptyset \text{\quad
		and\quad }\mathrm{Ext}(C_{\max }(I))\subset \mathrm{Ext}(C).
\end{equation*}
In particular, there exists $\bar{p}\in \mathrm{Ext}(C)$ such that
\begin{equation*}
	F_{i}(\bar{p})=\max_{p\in C}F_{i}(p),\text{ \quad for all }i\in I.
\end{equation*}
\end{lemma}

\begin{proof}
Since $F_{i}:(C,w^{\ast })\longrightarrow \mathbb{R}$ is usc, the set ${C}_{\max }(F_{i})$ is nonempty and $w^{\ast }$-compact in $\Phi ^{\ast }$. By
hypothesis,
\begin{equation*}
C_{\max }(I)=\bigcap_{i\in I}C_{\max }(F_{i})\neq \emptyset .
\end{equation*}
Since $S=C_{\max }(I)$ is nonempty $w^{\ast }$-compact and metrizable in 
$\Phi ^{\ast }$, applying Lemma~\ref{lem-ext} we deduce that 
\begin{equation*}
w^{\ast }\text{-}\mathrm{Exp}(S)\neq \emptyset .
\end{equation*}
It remains to show that $\mathrm{Ext}(S)\subset \mathrm{Ext}(C).$ To
this end, let $\bar{p}\in \mathrm{Ext}(S)$ and assume, towards a
contradiction, that there exists $p_{1},p_{2}\in C\setminus \{\bar{p}\}$
such that $\bar{p}\in (p_{1},p_{2}).$ Since $\bar{p}\in \mathrm{Ext}(S)$, we
may assume with no loss of generality that $p_{1}\in C\setminus S$.
Therefore, there exists $i_{0}\in I$ such that $p_{1}\not\in {C}_{\max}(F_{i_{0}})$. It follows that 
\begin{equation*}
F_{i_{0}}(p_{1})<\max_{p\in C}F_{i_{0}}(p)=F_{i_{0}}(\bar{p})\text{ \quad
	and \quad }F_{i_{0}}(p_{2})\leq \max_{p\in C}f_{i_{0}}(p)=F_{i_{0}}(\bar{p}),
\end{equation*}
which contradicts the fact that $F_{i_{0}}$ is convex. Thus, $w^{\ast}$-$\mathrm{Exp}(S)\subset \mathrm{Ext}(S)\subset \mathrm{Ext}(C)$ and the conclusion follows.\end{proof}

\bigskip 

%%%%%%%%%%%%%%%%%%%%%
We are now ready to establish the following result which is a generalized version of Bauer's maximum principle. Roughly speaking, whenever
a family of usc convex-trace functions on $K$ has at least one common
maximizer, then a common maximizer can be found among the points of the
Choquet boundary of $K$.

\begin{theorem} \label{Pconvex1} Let $K$ be a compact metric space and $\Phi$ a closed subspace of $\mathcal{C}(K)$ that separates points in $K$ and contains the constant functions. Let further $\{f_{i}\}_{i\in I}\subset T\mathcal{C}^{>}(K,\Phi )$ be such that 
	\begin{equation}\label{eq:hak}
K_{\max }(I):=\bigcap_{i\in I}{K}_{\max }(f_{i})\neq \emptyset .
	\end{equation}
Then, there exists $\bar{x}\in \partial_{\Phi }(K)$ (Choquet boundary of $K$) such that
	\begin{equation*}
f_{i}(\bar{x})=\max_{x\in K}f_{i}(x),\quad \text{for every }i\in I.
	\end{equation*}
\end{theorem}

\begin{proof}
Since $K$ is compact metric space, $\mathcal{C}(K)$ is separable and so is
its closed subspace $\Phi$. It follows that the convex $w^{\ast }$-compact
subset $K(\Phi )$ of $\Phi ^{\ast }$ is metrizable. By Definition~\ref{def_convex-trace}, for each $i\in I$ there exists an usc convex function $F_{i}:(K(\Phi ),w^{\ast })\rightarrow {\mathbb{R}}$ such that $f_{i}=F_{i}\circ \delta ^{\Phi }$. Set $C:=(K(\Phi ),w^{\ast })$. By
Proposition~\ref{max}, $\delta ^{\Phi }(K_{\max }(I))\subset
\bigcap_{i\in I}C_{\max }(F_{i})$, therefore by \eqref{eq:hak}
\begin{equation*}
C(I):=\bigcap_{i\in I}C_{\max }(F_{i})\neq \emptyset .
\end{equation*}
Then Lemma~\ref{Pconv} yields the existence of a common maximizer 
\begin{equation*}
\bar{Q}\in \mathrm{Ext}(K(\Phi ))=\delta ^{\Phi }(\partial _{\Phi }(K))
\end{equation*}
for all usc convex functions $F_{i}$, $i\in I$. Therefore, there exists $\bar{x}\in \partial _{\Phi }(K)$ such that 
$\bar{Q}=\delta _{\bar{x}}^{\Phi}$. Since 
\begin{equation*}
\max_{x\in K}\,f_{i}(x)\,\leq \,\max_{Q\in K(\Phi)}F_{i}(Q)\,=\,F_{i}(\delta_{\bar{x}}^{\Phi })\,=\,f_{i}(\bar{x}),\quad 
\text{for all }i\in I,
\end{equation*}
we conclude that $\bar{x}\in \bigcap_{i\in I}K_{\max }(f_{i})$.
Therefore we conclude that $\bar{x}\in \partial _{\Phi }(K)$ is a common maximizer of all functions $f_{i}$, $i\in I$. \end{proof}

\bigskip
%%%%%%%%%%%%%%%%%%%%%%%%%5

We obtain the following characterization of the Choquet boundary of a
compact metric space.

\begin{corollary}[characterization of the Choquet boundary]\label{cor-char}
	Let $K$ be a compact metric space and $\Phi $ as in Theorem~\ref{Pconvex1}.
	Then
	\begin{equation*}
	\bar{x}\in \partial _{\Phi }(K)\Longleftrightarrow \{\bar{x}\}=\bigcap
	\left\{ K_{\max }(f):\bar{x}\in K_{\max }(f),f\in T\emph{C}^{>}(K)\right\} .
	\end{equation*}
\end{corollary}

\begin{proof}
Let us first assume that
\begin{equation}
\bigcap \left\{ K_{\max }(f):\bar{x}\in K_{\max }(f),f\in T\emph{C}%
^{>}(K)\right\} =\{\bar{x}\}.  \label{eq:hak2}
\end{equation}
Then setting
\begin{equation*}
\mathcal{F}:=\left\{ f\in T\emph{C}^{>}(K):\;\bar{x}\in K_{\max }(f)\right\}
\end{equation*}
we have $\bigcap_{f\in \mathcal{F}}K_{\max }(f)\neq \emptyset$. Therefore,
by Theorem~\ref{Pconvex1}, $\bigcap_{f\in \mathcal{F}}K_{\max }(f)\cap
\partial_{\Phi }(K)\neq \emptyset $ and consequently, $\bar{x}\in \partial_{\Phi }(K)$. \smallskip \newline
Conversely, let $\bar{x}\in \partial _{\Phi }(K)$, that is, $\delta _{\bar{x}%
}\in \mathrm{Ext}(K(\Phi ))$. We define the function $F_{\bar{x}}:(K(\Phi
),w^{\ast })\rightarrow {\mathbb{R}}$ by $F_{\bar{x}}(\delta _{\bar{x}})=1$
and $F_{\bar{x}}(Q)=0$ on $K(\Phi )\setminus \{\delta _{\bar{x}}\}$. Clearly 
$F_{\bar{x}}$ is convex and upper semicontinuous on $(K(\Phi ),w^{\ast })$.
It follows that $f_{\bar{x}}:=F_{\bar{x}}\circ \delta ^{\Phi }\in T\emph{C}^{>}(K)$ and $K_{\max }(f_{\bar{x}})=\{\bar{x}\}$, which readily yields \eqref{eq:hak2}. \end{proof}

%%%%%%%%%%%%%%%%%%%%%

\subsubsection{Generic maximum principle.}

In this section we shall establish a generic version of the maximum
principle. Similarly to Subsection~\ref{ssec:4.1}, we consider a compact
metric space $(K,d)$ and we assume that $\Phi $ is a closed subspace of 
$\mathcal{C}(K)$ that separates points in $K$ and contains the constant
functions. In this setting, the Banach space $\mathcal{C}(K)$ (and a
fortiori $\Phi $) is separable and the convex set $K(\Phi )\subset \Phi
^{\ast }$ is $w^{\ast }$-compact metrizable. We shall denote by $d_{\Phi }$
a metric on $K(\Phi )$ compatible with its $w^{\ast }$-topology. Then $(K(\Phi ),d_{\Phi })$ is also a compact metric space.\smallskip 

For any nonempty set $X$ we denote by $\mathbb{R}^{X}$ the space of all
real-valued functions on $X$ and by $\rho _{\infty }$ the (complete) metric;
\begin{equation}
\rho _{\infty }(f,g):=\sup_{x\in X}\frac{|f(x)-g(x)|}{1+|f(x)-g(x)|},\quad\text{for all } f,g\in \mathbb{R}^{X}.  \label{eq:rho}
\end{equation}

In what follows, we set 
\begin{equation*}
\widehat{\Phi }:=\{\widehat{\phi }\in \Phi ^{\ast \ast }:\phi \in \Phi
\},\quad \text{where\  }\widehat{\phi }(Q)=\langle Q,\phi \rangle ,\;\text{
	for all }Q\in \Phi^{\ast}.
\end{equation*}
Recalling terminology from Definition~\ref{def_convex-trace}, and dropping
dependence on $\Phi $ in the notation of convex-trace functions, we have: 
\begin{equation*}
\Phi \subset \Gamma _{\Phi }(K)=T\mathcal{C}(K)\subset T\mathcal{C}%
^{>}(K)\quad \text{and}\quad \text{\ }\widehat{\Phi }\subset \Gamma (K(\Phi
))\subset \Gamma ^{>}(K(\Phi )).
\end{equation*}

We also recall from Remark~\ref{rem-Phi-stable} that a subset $\mathcal{B}
$ of $T\mathcal{C}^{>}(K)$ is called $\Phi $-stable, if $\Phi +\mathcal{B}\subset \mathcal{B}$. Examples of $\Phi $-stable subsets of $T\mathcal{C}^{>}(K)$ are $\Phi$ itself, the class of Choquet convex
functions $\Gamma _{\Phi }(K)=T\mathcal{C}(K)$ and the set of usc
convex-trace functions $T\mathcal{C}^{>}(K)$. \smallskip \newline
 Let us finally notice that Theorem~\ref{Pconvex1} (applied to a
family of one element) yields that if a function $f\in T\mathcal{C}^{>}(K)$
has a unique maximizer, then this maximizer belongs to the Choquet boundary
of~$K$. We are now ready to state the main result of this section. 

\begin{theorem}[genericity of unique maximizer]
	\label{Choquet} Let $(K,d)$ be a compact metric space and $\Phi $ be a
	closed subspace of $C(K)$ which separates the points of $K$. Let $\mathcal{B}
	$ be a $\Phi $-stable subset of $T\mathcal{C}^{>}(K)$. Then, the set 
	\begin{equation*}
	\mathcal{G}:=\left\{ f\in \mathcal{B}:K_{\max }(f)\,\text{is a singleton}\,\right\} 
	\end{equation*}
	is a $G_{\delta }$ dense subset in $(\mathcal{B},\rho _{\infty })$.
\end{theorem}

\begin{proof}
Let us denote by $\rho _{\infty }$ the metric of uniform convergence on both 
$\mathbb{R}^{K(\Phi )}$ and $\mathbb{R}^{K}.$ Since uniform limits maintain
upper semicontinuity and convexity, it follows that the metric space $(\Gamma ^{>}(K(\Phi )),\rho _{\infty })$ is complete (as a closed subspace
of $(\mathbb{R}^{K(\Phi )},\rho _{\infty })$). A standard argument now shows
that $(T\mathcal{C}^{>}(K),\rho _{\infty })$ is closed in $(\mathbb{R}^{K},\rho _{\infty })$ and consequently, it is also complete. \smallskip 

We now consider the following canonical map $\mathcal{S}:(\Gamma ^{>}(K(\Phi )),\rho _{\infty }) \longrightarrow (T\mathcal{C}^{>}(K),\rho _{\infty })$ defined by $\mathcal{S}(\widehat{F}) =\widehat{F}\circ \delta ^{\Phi}$.
It is easily seen that $\mathcal{S}$ is surjective and $1$-Lipschitz. Moreover, it
is easy to see that there exists $\mathcal{A}\subset \Gamma^{>}(K(\Phi ))$ such that $\mathcal{B}=\mathcal{S}(\mathcal{A})$ and 
$\widehat{\Phi}+\mathcal{A}\subset \mathcal{A}$ (that is, $\mathcal{A}$ is a 
$\widehat{\Phi }$-stable subset of $\Gamma^{>}(K(\Phi))$).\smallskip 

\textit{Claim.} The set $\mathcal{D}:=\left\{ \widehat{F}\in \mathcal{A}:\widehat{F}\,\text{has unique maximum on }K(\Phi) \right \}$
is $G_{\delta }$ dense in $(\mathcal{A},\rho _{\infty })$. \medskip 

\noindent\textit{Proof of the Claim.} For $n\geq 1$, we set:
\begin{equation*}
\widehat{\mathcal{U}}_n:=\bigg \{\widehat{F}\in \mathcal{A};\;\exists
Q_{n}\in K(\Phi )\text{ \ with }\widehat{F}(Q_{n})>\underset{Q\in K(\Phi
	):\,d_{\Phi }(Q,Q_{n})\geq \frac{1}{n}}{\sup }\widehat{F}(Q)\bigg \}.
\end{equation*}
It is easy to see that $\widehat{\mathcal{U}}_n$ is an open subset of $(\mathcal{A},\rho _{\infty })$ for all $n\geq 1$, and $\mathcal{D}:=\bigcap_{n\geq 1}\,\widehat{\mathcal{U}}_n.$
Thanks to Lemma~\ref{MBAD} (applied to the compact metric space $(K(\Phi
),d_{\Phi })$ and the subspace $(\widehat{\Phi },\Vert \cdot \Vert _{\infty
})$ of $\mathcal{C}(K(\Phi ))$), for every $0<\varepsilon <1$ and $\widehat{F}\in \mathcal{A}$, there exists a function $\phi \in \Phi $ such that $\rho
_{\infty }(\widehat{\phi },0)<\varepsilon $ and $-\widehat{F}-\widehat{\phi}
$ attains a unique minimum on $K(\Phi ).$ Let us denote by $Q_{0}\in K(\Phi)$ this unique minimizer. Then we deduce that $\widehat{G}:=\widehat{F}+
\widehat{\phi }\in \bigcap_{n\geq 1}\widehat{\mathcal{U}}_n$ (we
take $Q_{n}=Q_{0},$ for all $n\geq 1$) and $\rho _{\infty }(\widehat{G},
\widehat{F})<\varepsilon $. Thus the $G_{\delta }$-set $\mathcal{D}$ is
dense in $(\mathcal{A},\rho_{\infty })\ $and the claim follows.\medskip 

From the classical Bauer theorem, the unique maximizer of every usc convex
function $\widehat{F}\in \mathcal{D}\subset \Gamma ^{>}\mathrm{\,}(K(\Phi))$
is necessarily an extreme point of $K(\Phi )$. Then by Proposition~\ref{max}
we deduce that for every $\widehat{F}\in \mathcal{D}$, the usc convex-trace
function $\mathcal{S}(\widehat{F})=\widehat{F}\circ \delta ^{\Phi }$ has a
unique maximum on $K,$ which is necessarily attained at a point in the
Choquet boundary $\partial _{\Phi }(K)$. Since 
\begin{equation*}
\mathcal{D}\subset \left\{ \widehat{F}\in \mathcal{A}:K_{\max }\left( 
\mathcal{S}(\widehat{F})\right) \mathnormal{\ }\text{singleton}\right\} ,
\end{equation*}
we obtain readily that $\mathcal{S}(\mathcal{D})\subset \mathcal{G}$. Since 
$\mathcal{S}$ is a Lipschitz surjective map and $\mathcal{D}$ is dense in $(\mathcal{A},\rho _{\infty }),$ we deduce that $\mathcal{G}$ is dense in 
$(\mathcal{B},\rho _{\infty })$. Moreover $\mathcal{G}$ is a $G_{\delta }$
subset of $(T\mathcal{C}^{>}(K),\rho _{\infty })$ since it can be written as 
$\mathcal{G}=\bigcap_{n\geq 1}\mathcal{U}_{n}$, where
\begin{equation*}
\mathcal{U}_{n}:=\bigg \{f\in \mathcal{B};\;\exists x_{n}\in K\;f(x_{n})\,>\,
\underset{x\in K:\,d(x,x_{n})\geq \frac{1}{n}}{\sup }\,f(x)\bigg \},
\end{equation*}
is open in $(\mathcal{B},\rho _{\infty })$ for each $n\geq 1$. This
completes the proof. \end{proof}

\bigskip 

It follows from the above result, by taking $\mathcal{B}=T\mathcal{C}^{>}(K)$,
that a generic upper semicontinuous convex-trace function on $K$ attains a unique maximum (necessarily at a point of the Choquet boundary). By taking $\mathcal{B}= \Gamma_{\Phi }(K)$, we obtain the same conclusion for a generic Choquet-convex function. Both results are new and together with Theorem~\ref{Pconvex1} provide generalized version of the classical Bauer maximum principle. 

%%%%%%%%%%%%%%%%%%%%%%%%%%%%%%%%%%%%%%%%%
\section{Examples.}
\label{sec:ex}

In this section we provide three examples-schemes to illustrate this theory. In the first example (Subsection~\ref{ssec:5.1}) we show how the Choquet boundary of the closed interval $[0,1]$ may change depending on the choice of $\Phi$. In particular, every closed subset of $[0,1]$ that contains the extreme points $\{ 0,1 \}$ can be identified to the Choquet boundary of $[0,1]$ under an adequate choice of the space $\Phi\subset\mathcal{C}([0,1])$. In the second example (Subsection~\ref{ssec:5.2}) we deal with the unit disk of the complex plane. Then the class of convex-trace functions consists of the subharmonic functions, if $\Phi$ is taken to be the harmonic functions. We describe the convex-traces subsets of the disk using Runge's approximation theorem as well as the maximum principle for harmonic functions. Finally, in Subsection~\ref{ssec:5.3} we illustrate trace-convexity for subsets of the set of natural numbers $\mathbb{N}$ (with its discrete topology).

\subsection{Choquet boundaries of $[0,1]$.}\label{ssec:5.1}

Let us first notice that in view of Definition~\ref{phi_extrem} and relation~\eqref{eq:Ch-boundary}, the set $\mathrm{Ext}_{\Phi}(K)$ of $\Phi$-extreme points of $K$ coincides with the $\Phi$-Choquet boundary of $K$ (see also the proof of Corollary~\ref{cor-ch-bd}). \smallskip

We shall now describe the $\Phi$-Choquet boundary of the closed interval $K=[0,1]$, under various choices of closed subspaces $\Phi \subset \mathcal{C}([0,1])$.
\smallskip

(i). Let us first consider the case $\Phi=\mathrm{Aff}([0,1])$. In this case we have (cf. Remark~\ref{rem-compat-1}) 
$$\partial_{\Phi}([0,1])=\partial ([0,1])=\{0,1\} \quad \text{and} \quad \Gamma_{\Phi}([0,1])=\Gamma ([0,1]),$$ that is, the Choquet boundary coincides with the usual boundary and the class of Choquet functions coincides with the class of convex continuous functions on~$[0,1]$.
	\smallskip
		
(ii). Let $\Delta \subset \lbrack 0,1]$ be the usual Cantor set. Then $[0,1]\setminus \Delta =\bigcup_{n\geq 1}(a_{n},b_{n})$ with 
$\{a_{n}\}_{n}$, $\{b_{n}\}_{n}\subset \Delta $. Then defining 
	\begin{equation*}
	\Phi _{\Delta }:=\left\{ \phi \in \mathcal{C}([0,1]:\;\phi |_{[a_{n},b_{n}]} \text{ affine, for all }n\geq 1\right\} ,
	\end{equation*}
	we obtain that the Cantor set $\Delta$ becomes the $\Phi _{\Delta}$-Choquet boundary of $[0,1],$ that is, 
	\begin{equation*}
	\partial _{\Phi _{\Delta }}([0,1])=\Delta .
	\end{equation*}
Indeed, if $x\in \lbrack 0,1]\setminus \Delta ,$ then there exists 
$n_{0}\geq 1$ such that $x\in (a_{n_{0}},b_{n_{0}})$ and consequently 
$x=ta_{n_{0}}+(1-t)b_{n_{0}},$ for some $t\in (0,1)$. We set 
\begin{equation*}
\mu :=t\delta _{a_{n_{0}}}+(1-t)\delta _{b_{n_{0}}}\in \mathcal{M}^{1}([0,1]).
\end{equation*}
	Since every $\phi \in \Phi _{\Delta }$ is affine on $[a_{n},b_{n}]$ we obtain
	\begin{equation*}
	\langle \mu .\phi \rangle =t\phi (a_{n_{0}})+(1-t)\phi (b_{n_{0}})=\phi
	(x)=\langle \delta _{x}.\phi \rangle ,
	\end{equation*}
yielding $\mu \sim \delta_{x}$ and consequently $x\notin \partial_{\Phi_{\Delta }}([0,1])$. On the other hand, if $x\in \Delta$, then $x$ is the unique maximum of the function $\phi\in\Phi_{\Delta}$ defined by $\phi (t)=-|t-x|,$ $t\in \lbrack 0,1]$. Therefore, $x\in \partial_{\Phi_{\Delta }}([0,1])$.\medskip

\noindent\textit{Remark.} The above proof works in the same way for any closed subset $F$ of $[0,1]$ that contains the extreme points $\{0,1\}$. Therefore, any such closed subset is the $\Phi$-Choquet boundary of $[0,1]$ under an adequate choice of $\Phi$.\medskip

(iii). Let us now take $\Phi =\mathcal{C}[0,1]$. Then every point of $[0,1]$ belongs to the Choquet-boundary (i.e. $\partial_{\Phi}([0,1])=[0,1]$) and every continuous function is $\Phi$-Choquet convex.

\subsection{Harmonic functions on the disk.}
\label{ssec:5.2}

In this subsection we shall deal with harmonic functions on the unit disk. For prerequisites in complex analysis, as well as for notions and results that will be evoked in this section we refer the reader to~\cite{GreenKrantz}.\smallskip 

Let $\mathbb{D}:=\{z\in \mathbb{C}:\;|z|<1\}$ denote the open disk of the
complex plane and $K=\mathbb{\bar{D}}$ its closure. We set:
\begin{equation*}
\Phi =\{u:\mathbb{\bar{D}}\rightarrow \mathbb{R}\,\;\text{continuous, }u|_{\mathbb{D}}\;\text{harmonic}\}.
\end{equation*}
In this case, the class of Choquet-convex functions 
$\Gamma_{\Phi}(K)=T\mathcal{C}(K)$ coincides with the continuous subharmonic functions on $\mathbb{D}$, while $T\mathcal{C}^{>}(K)$ corresponds to the class of upper semicontinuous subharmonique functions.\smallskip 

Let us now recall that subharmonique functions satisfy the \emph{maximum principle}, that is, the maximum of any subharmonic function $u$ over any
set is attained at the boundary of the set. Combining this with Corollary~\ref{cor-char} we deduce that the Choquet boundary 
$\partial _{\Phi }\mathbb{\bar{D}}$ of $\mathbb{\bar{D}}$ is contained in $\partial \mathbb{\bar{D}}:=\{z\in \mathbb{C}:|z|=1\}.$ On the other hand, any $\bar{z}\in\partial \mathbb{D}$ is the unique maximizer of some $\phi \in \Phi $ (to see this, it suffices to take any $h\in \mathcal{C}(\partial \mathbb{\bar{D}})$ with strict maximum at $\bar{z}$ and obtain $\phi$ as the unique solution of the Laplace equation $\Delta u=0$ on $\mathbb{D}$ with $u|_{\partial \mathbb{\bar{D}}}=h$. Therefore, in view of Theorem~\ref{Pconvex1} (applied to one function) we deduce:
\begin{equation*}
\partial _{\Phi }\mathbb{\bar{D}}=\partial \mathbb{D}=\{z\in \mathbb{C}:\,|z|=1\}.
\end{equation*}

We shall now describe the convex-trace subsets of $\mathbb{D}$. We shall need the following lemma.

\begin{lemma}
	\label{lem-Ubar}Let $\mathcal{U}$ be a nonempty open simply connected subset of $\mathbb{D}$ such that $\mathcal{\bar{U}\subset }\mathbb{D}$. Then the (closed) set $\mathcal{\bar{U}}$ is convex-trace in $\mathbb{D}$.
\end{lemma}

\begin{proof} We shall use the characterization of trace-convexity given in
Corollary~\ref{Cor_separation}. To this end, let 
$\bar{x}\in \mathbb{\bar{D}}\diagdown \mathcal{\bar{U}}$. We distinguish two cases.\smallskip 

\textit{Case 1}. $\bar{x}\in \partial \mathbb{D}$. \smallskip 

In this case, using the same standard argument that we evoked before, we deduce that $\bar{x}$ is
the unique maximizer of some $\phi \in \Phi $ and consequently $\sup_{x\in 
	\mathcal{\bar{U}}}\,\phi (x)<\phi (\bar{x})$. 

\medskip 

\textit{Case 2}. $\bar{x}\in \mathbb{D\diagdown}\mathcal{\bar{U}}$.\smallskip 

In this case, there exists $r>0$ such that $\bar{B}(\bar{x},r)\subset 
\mathbb{D}$ and $\bar{B}(\bar{x},r)\cap \mathcal{\bar{U}}=\emptyset$. Then 
$\bar{B}(\bar{x},r)\cup \mathcal{\bar{U}}$ is a compact simply connected
subset of $\mathbb{D}$. Let $\mathcal{V}_{1},\mathcal{V}_{2}$ be disjoint
open simply connected subsets of $\mathbb{D}$ such that $\bar{B}(\bar{x},r)\subset \mathcal{V}_{1}$ and $\mathcal{\bar{U}}\subset \mathcal{V}_{2}$.
Then the function
\begin{equation*}
\varphi (z)=\left\{ 
\begin{array}{cc}
0, & \text{if }z\in \mathcal{V}_{1} \\ 
1, & \text{if }z\in \mathcal{V}_{2}\end{array}
\right. 
\end{equation*}
is (trivially) holomorphic on $\mathcal{V}_{1}\cup \mathcal{V}_{2}$.
Applying Runge's approximation theorem we deduce that there exists a complex polynomial $f(z)$ such that 
\begin{equation*}
||f-\varphi ||_{\infty }:=\sup_{z\in \bar{B}(\bar{x},r)\cup \mathcal{\bar{U}}
}|f(z)-\varphi (z)|<\frac{1}{3}.
\end{equation*}
Taking $u=\mathrm{Re}(f)$ we deduce easily that $u\in \Phi $ (harmonic) and 
$\rho _{\infty }(u,\varphi )<1/3.$ It follows that $\sup_{x\in \mathcal{\bar{U
}}}\,u(x)<1/3$ and $u(\bar{x})>2/3$. Therefore (\ref{eq:sep}) holds and 
$\mathcal{\bar{U}}\in \mathcal{P}_{TC}(\mathbb{D}).$ \end{proof}

\bigskip 
We shall now provide a nice description of convex-trace sets of $\mathbb{D}$. 

\begin{proposition}[convex-trace sets of the complex disk]
	Let\ $C$ be a nonempty, compact subset of $\mathbb{D}$ with finitely many
	connected components. Then 
	\begin{equation*}
	C\in \mathcal{P}_{TC}(X)\quad \Longleftrightarrow \quad C\text{ is simply
		connected}
	\end{equation*}
\end{proposition}

\begin{proof} Let $C\neq \emptyset $ be compact in $\mathbb{D}$. We shall
first show that ``being simply connected" is a necessary condition for
trace-convexity. Indeed, assume towards a contradiction that $\mathbb{D}\diagdown C$ is not path connected and let $\bar{x}\in \mathbb{D}\diagdown C$
be a point that cannot be joined to the boundary $\partial \mathbb{D}$ via a
continuous path lying in $\mathbb{D}\diagdown C$. Then $\bar{x}$ is
surrounded by a curve lying in $C$ and (\ref{eq:sep}) fails for all $\phi
\in \Phi $ in view of the maximum principle for harmonic
functions.\smallskip 

Let us now assume that $C$ is simply connected and let 
$\bar{x}\in \mathbb{\bar{D}}\diagdown C$. Similarly to the proof of Lemma~\ref{lem-Ubar} (Case~1) we may assume that $\bar{x}\in \mathbb{D}$. Since $C$ is compact and is
contained in $\mathbb{D}$, there exists $r>0$ sufficiently small, such that
the $r$-enlargement $C_{r}:=C+B(0,r)$ is open, simply connected and its
closure $\bar{C}_{r}:=C+\bar{B}(0,r)$ remains in $\mathbb{D}$. Shrinking
further $r$ if necessary, we may assume that $B(\bar{x},r)\subset \mathbb{D}$
and $B(\bar{x},r)\cap C_{r}=\emptyset .$ Notice that (the open set) $C_{r}$
has finitely many components and that each connected component is open and
simply connected. We denote by $\{\mathcal{V}_{i}\}_{i=1}^{k}$ the connected
components of $C_{r}\ $and we define $\varphi :C_{r}\cup B(\bar{x},r)\rightarrow \mathbb{C}$ by $\varphi |_{\mathcal{V}_{i}}\equiv i$ for $i\in \{1,\ldots ,k\}$ and $\varphi |_{B(\bar{x},r)}\equiv k+1$. Then 
$\varphi $ is trivially a holomorphic functions on $C_{\varepsilon }$ for any 
$\varepsilon \in (0,r)$. Then by Runge's approximation theorem we deduce the
existence of a complex polynomial $f\in H(\mathbb{D)}$ with 
\begin{equation*}
||f-\varphi ||_{\infty }=\sup_{z\in \bar{B}(\bar{x},\varepsilon )\cup
	C_{\varepsilon }}|f(z)-\varphi (z)|<\frac{1}{3}.
\end{equation*}
Taking $u=\mathrm{Re}(f)\in \Phi $ we conclude that 
\begin{equation*}
\sup_{x\in C}\,u(x)<k+\frac{1}{3}\text{\quad and}\quad u(\bar{x})>k+\frac{2}{3}
\end{equation*}
which yields the result. \end{proof} 

%%%%%%%%%%%

\subsection{Trace-convexity for subsets of $\mathbb{N}$.}
\label{ssec:5.3}

Let $X=\mathbb{N}$ be the set of natural numbers, viewed as a completely
regular topological space with its discrete topology. Let further $K=\beta 
\mathbb{N}$. Then 
\begin{equation*}
C_{b}(\mathbb{N})=C(\beta \mathbb{N})=\ell ^{\infty }(\mathbb{N})=\left\{
y=\{y_{n}\}_{n}:\,\,\sup_{n\geq 1}\,|y_{n}|\,:=||y||_{\infty }<+\infty
\right\} .
\end{equation*}
Let $\boldsymbol{1}$ denote the constant sequence with all coordinates equal
to $1$, and $\boldsymbol{b}=\{b_{n}\}_{n}$ with $b_{n}=1/n$ for all $n\in 
\mathbb{N}$. We take 
\begin{equation*}
\Phi =\mathrm{span}\{\boldsymbol{1},\boldsymbol{b}\}\quad \text{(}2\text{-dimensional subset of }\ell ^{\infty }(\mathbb{N})\text{).}
\end{equation*}
Then $\Phi $ obviously contains the constant functions (constant sequences).
It also separates points in $\mathbb{N}$ since the sequence $\boldsymbol{b}$
is injective. Notice that $\Phi $ can be isometrically identified to $\mathbb{R}^{2}$ by identifying $c=(c_{1},c_{2})$ to the (bounded) sequence $\widehat{c}=\left( c_{1}+\frac{c_{2}}{n}\right)_{n\geq 1}$ and by equipping $\mathbb{R}^{2}$ with the following norm:
\begin{equation*}
||c||_{\Phi }:=||\widehat{c}||_{\infty }=\max \{|c_{1}|,|c_{1}+c_{2}|\}.
\end{equation*}
The positive cone of $\Phi $ (cone of positive sequences of $\Phi $)
corresponds to the cone
\begin{equation*}
\Phi _{+}=\left\{ c=(c_{1},c_{2})\in \mathbb{R}^{2}:\,c_{1}\geq 0,
c_{1}+c_{2}\geq 0\right\} \subset (\mathbb{R}^{2},||.||_{\Phi }),
\end{equation*}
with polar cone
\begin{equation*}
\Phi _{+}^{\ast }:=\left\{ Q=(Q_{1},Q_{2})\in \mathbb{R}^{2}:\,Q_{1}\geq
Q_{2}\geq 0\right\} \subset (\mathbb{R}^{2},||.||_{\Phi ^{\ast }})
\end{equation*}
where 
\begin{equation*}
||Q||_{\Phi ^{\ast }}=\left\{ 
\begin{array}{cc}
\max \{|Q_{1}|,|Q_{2}|\}, & \text{if }Q_{1}Q_{2}>0 \medskip 
\\ 
|Q_{1}|+|Q_{2}|, & \phantom{i}\text{if } Q_{1}Q_{2}<0.
\end{array}
\right. 
\end{equation*}
Then~\eqref{eq: K(Phi)} yields:
\begin{equation*}
K(\Phi )=\Phi _{+}^{\ast }\cap \{Q\in \Phi ^{\ast }:||Q||_{\Phi ^{\ast
}}=1\}=\{Q\in \Phi _{+}^{\ast }:\langle Q,\boldsymbol{1}\rangle =1\}=\left\{
\,(1,1+t):\,t\in \lbrack 0,1]\right\} .
\end{equation*}
Thus, $\mathrm{Ext\,}K(\Phi )=\{(1,0),(1,1)\}$. According to~\eqref{eq: detla_Phi}, the canonical injection $\delta ^{\Phi }:\mathbb{N}
\rightarrow K(\Phi )$ gives for $k\in \mathbb{N}:$  
\begin{equation*}
\delta ^{\Phi }(k)=\widehat{k}\quad \text{with\quad }\widehat{k}(\widehat{c}
)=\widehat{c}(k)=c_{1}+\frac{c_{2}}{k}=\left\langle (1,\frac{1}{k}
),(c_{1},c_{2})\right\rangle.
\end{equation*}
Therefore $\delta ^{\Phi }(\mathbb{N})=\{(1,1/k):k\geq 1\}\subset K(\Phi )$
and the $\Phi $-Choquet boundary of $\mathbb{N}$ is $\partial _{\Phi }
\mathbb{N}=\{1,\infty \}$ where $\infty \in \beta X$ is the ultrafilter
generated by the sets $A_{n}=\{k:k\geq n\}.$ It follows directly from
Definition \ref{def_TC-set} that a subset $A\subset \mathbb{N}$ is
trace-convex if and only if it is an \textit{interval }with respect to the
order of $\mathbb{N}$, that is, if it is of the form 
\begin{equation*}
\{k:n_{1}\leq k\leq n_{2}\},\text{ }\{k:k\leq n_{2}\}\quad \text{or }%
\{k:k\geq n_{1}\}
\end{equation*}
for some $n_{1},n_{2}\in \mathbb{N}$. Finally, a function (sequence) 
$y=\{y_{n}\}_{n\geq 1}\in \ell ^{\infty }$ is Choquet convex (according to
Corollary \ref{equiv} and Definition~\ref{def_convex-trace}) if and only if 
$y_{n}=f(1/n)$ for some convex function $t\mapsto f(t)$ defined on 
$[0,1]\approx K(\Phi )$.

\medskip 

Let us finally mention that different interesting notions of trace-convexity
can be obtained by taking $\boldsymbol{b}=(b_{n})_{n}$ to be (injective and)
non-monotone (for instance, $b_{n}=1+(-1)^{n}/n,$ for all $n\geq 1$) or by
completely different choices of $\Phi $ (of dimension either~$2$ or more).
The interesting reader might see how these choices modify the Choquet
boundary of $\mathbb{N}$ in relation with the results of Section~\ref{sec:4}.

\bigskip 

%%%%%
\hrule

\bigskip

\noindent\textbf{Acknowledgement.}\, Part of this work has been conducted within the FP2M federation (CNRS FR 2036) and during a research visit of the second author to the SAMM Laboratory of the University Paris Sorbonne-Panthéon (February 2020). This author thanks his hosts for hospitality.

\bigskip

\hrule
%\bigskip
\vspace{0.6cm}

\noindent Mohammed BACHIR\medskip

\noindent Laboratoire SAMM 4543, Universit\'{e} Paris 1
Panth\'{e}on-Sorbonne\newline Centre PMF, 90 rue Tolbiac, F-75634 Paris cedex
13, France.\smallskip

\noindent E-mail: \texttt{Mohammed.Bachir@univ-paris1.fr}

\noindent \texttt{http://samm.univ-paris1.fr/Mohammed-Bachir}

%\bigskip
\vspace{0.7cm}

\noindent Aris DANIILIDIS\medskip

\noindent DIM--CMM, UMI CNRS 2807\newline Beauchef 851, FCFM, Universidad de
Chile \\
CP 8370459, Santiago, Chile\smallskip

\noindent E-mail: \texttt{arisd@dim.uchile.cl}

\noindent \texttt{http://www.dim.uchile.cl/\symbol{126}arisd/}

\medskip

\noindent Research supported by the grants:\smallskip \\ 
CMM-CONICYT AFB170001, ECOS-CONICYT C18E04, FONDECYT
1171854 (Chile) \smallskip \\ and PGC2018-097960-B-C22 (MICINN, Spain and ERDF, EU).


\begin{thebibliography}{99}  %%% 

\bibitem {AB}\textsc{C. D. Aliprantis, K. C. Border,} \emph{Infinite
dimensional analysis. A hitchhiker's guide} (3rd ed.). (Springer, 2006).

\bibitem {Ba}\textsc{M. Bachir,} On the Krein-Milman-Ky Fan theorem for convex
compact metrizable sets, \emph{Illinois J. Math.} \textbf{62} (2018), 1--24.

\bibitem {Ba1}\textsc{M. Bachir} A non convexe analogue to Fenchel duality, \emph{J.
Funct. Anal.} \textbf{181} (2001), 300--312.

\bibitem {Bau}\textsc{H. Bauer,} Minimalstellen von Funktionen und
Extremalpunkte II, \emph{Arch der Math}. \textbf{11} (1960), 200--205.

\bibitem {CO}\textsc{B. Cascales, J. Orihuela,} On Compactness in Locally
Convex Spaces, \emph{Math. Z.} \textbf{195} (1987), 365--381.

\bibitem {ChMe}\textsc{G. Choquet, P.-A. Meyer,} Existence et unicit\'{e} des
representations int\'{e}grals dans les convexes compacts quelconques,
\emph{Ann. Inst. Fourier} (Grenoble) \textbf{13} (1963), 139--154.

\bibitem {DGZ}\textsc{R. Deville, G. Godefroy, V. Zizler,} A smooth
variational principle with applications to Hamilton-Jacobi equations in
infinite dimensions, \emph{J. Funct. Anal.} \textbf{111} (1993), 197--212.

\bibitem {DR}\textsc{R. Deville, J. P. Revalski,} Porosity of ill-posed
problems, \emph{Proc. Amer. Math. Soc.} \textbf{128} (2000), 1117--1124.

\bibitem {Ky}\textsc{K. Fan,} On the Krein-Milman theorem, \emph{Proceedings
of symposia in pure mathematics}, vol.7, Amer. Math. Soc., Providence, RI,
1963, pp. 211--219.

\bibitem{GreenKrantz} \textsc{R. Greene, S. Krantz}, \emph{Function
	Theory of One Complex Variable} (3rd Ed), Graduate Studies in Mathematics 
\textbf{40}, (AMS, 2006). 

\bibitem{Gr} \textsc{M. W. Grossman,} Relations of a paper of Ky Fan to a
theorem of Krein-Milman type, \emph{Math. Z.} \textbf{90} (1965), 212--214.

\bibitem{Kelley} \textsc{J. Kelley}, \emph{General Topology}, Graduate
Texts in Mathematics \textbf{27}, (Springer, 1955).
\bibitem {LNV}\textsc{J. Lukes, I. Netuka, J. Vesely, } Choquet's
theory and the Dirichlet problem, \emph{Expo. Math.} \textbf{20} (2002) 229--254.

\bibitem {Kh}\textsc{B. D. Khanh,} Sur la $\Phi$-Convexit\'{e} de Ky Fan,
\emph{J. Math. Anal. Appl.} \textbf{20} (1967), 188--193.


\bibitem {KM}\textsc{M. Krein, D. Milman, }On extreme points of regular convex
sets, \emph{Studia Math.} \textbf{9} (1940), 133--138.

\bibitem {Ph}\textsc{R. R. Phelps,} \emph{Lectures on Choquet's theorem} (2nd
ed.), Lecture Notes in Mathematics, vol. \textbf{1757} (Springer, 1997).

\bibitem {Ru}\textsc{W. Rudin,} \emph{Functional analysis} (2nd ed.)
International Series in Pure and Applied Mathematics, McGraw-Hill, (New York, 1991)
\end{thebibliography}
\end{document}